\documentclass[letterpaper,11pt,reqno]{amsart}

\usepackage{url}
\usepackage[colorlinks=true, linkcolor=blue, citecolor=ForestGreen]{hyperref}
\usepackage{cite}

\usepackage{latexsym,amsmath,amssymb,amsfonts,amsthm}
\usepackage{mathrsfs}
\usepackage{euscript}
\usepackage{mathtools}
\usepackage{dsfont}
\usepackage{bbm}
\usepackage{soul}     
\usepackage{comment}
\usepackage[utf8]{inputenc}

\usepackage[usenames, dvipsnames]{color}
\usepackage[usenames, dvipsnames, cmyk]{xcolor}
\usepackage{graphicx}
\usepackage[scriptsize,hang,raggedright]{subfigure}
\usepackage{multirow}
\usepackage{enumerate}
\usepackage[inline]{enumitem}

\usepackage{a4wide}
\usepackage{epsfig,epstopdf}

\numberwithin{equation}{section}

\definecolor{grey}{rgb}{.7,.7,.7}
\definecolor{refkey}{gray}{.45}
\definecolor{labelkey}{gray}{.45}

\newcommand{\xupref}[2]{\hspace{-0.3ex}\stackrel{\eqref{#1}}{#2}} 

\newtheorem{theorem}{Theorem}[section]
\newtheorem{proposition}[theorem]{Proposition}
\newtheorem{lemma}[theorem]{Lemma}

\newtheorem{corollary}[theorem]{Corollary}

\theoremstyle{remark}
\newtheorem{remark}[theorem]{Remark}
\theoremstyle{definition}
\newtheorem{definition}[theorem]{Definition}
\newtheorem{example}[theorem]{Example}

\usepackage{tikz}
\usepackage{pgf,pgfplots}
\usetikzlibrary{calc}
\usetikzlibrary{decorations.markings}
\usetikzlibrary{decorations.pathmorphing}
\usetikzlibrary{decorations.shapes}
\usetikzlibrary{shapes,arrows,shapes.geometric,patterns,fadings}

\newcommand{\N}{\mathbb N}

\newcommand{\R}{\mathbb R}

\newcommand{\dd}{\,\mathrm{d}}

\newcommand{\defeq}{\coloneqq}

\newcommand{\ba}{\begin{array}}
\newcommand{\ea}{\end{array}}

\newcommand{\tld}[1]{\widetilde{#1}}
\newcommand{\bthm}{\begin{theorem}}
\newcommand{\ethm}{\end{theorem}}
\newcommand{\bprop}{\begin{proposition}}
\newcommand{\eprop}{\end{proposition}}
\newcommand{\blemma}{\begin{lemma}}
\newcommand{\elemma}{\end{lemma}}
\newcommand{\bexmpl}{\begin{example}}
\newcommand{\eexmpl}{\end{example}}

\newcommand{\beqn}{\begin{equation}}
\newcommand{\eeqn}{\end{equation}}
\newcommand{\beqns}{\begin{equation*}}
\newcommand{\eeqns}{\end{equation*}}

\newcommand{\supp}{\operatorname{supp}}

\newcommand{\ol}{\overline}

\newcommand{\Prb}{\mathbb{P}}

\renewcommand{\leq}{\leqslant}
\renewcommand{\geq}{\geqslant}

\definecolor{mygreen}{rgb}{0.1,0.75,0.2}

\newcommand{\E}{\mathbb{E}}

\newcounter{myenumi}
\DeclareMathOperator*{\esssup}{ess\,sup}

\title[Stochastic Keller--Segel type Equation]{On global existence and blowup of solutions of Stochastic Keller--Segel type Equation}

\author{Oleksandr Misiats}
\address[Oleksandr Misiats]{Department of Mathematics and Applied Mathematics, Virginia Commonwealth University, Richmond, VA, USA}
\email{omisiats@vcu.edu}

\author{Oleksandr Stanzhytskyi}
\address[Oleksandr Stanzhytskyi]{Department of Mathematics, Taras Shevchenko National University of Kiev, Ukraine}
\email{ostanzh@gmail.com}

\author{Ihsan Topaloglu}
\address[Ihsan Topaloglu]{Department of Mathematics and Applied Mathematics, Virginia Commonwealth University, Richmond, VA, USA}
\email{iatopaloglu@vcu.edu}

\date{\today}                                        

\thanks{This is a post-peer-review, pre-copyedit version of an article published in NoDEA. The final
authenticated version is available online at: \url{http://dx.doi.org/10.1007/s00030-021-00735-2}.}                                    

\begin{document}

\begin{abstract}
In this paper we consider a stochastic Keller--Segel type equation, perturbed with random noise.  We establish that for special types of random pertubations (i.e. in a divergence form), the equation has a global weak solution for small initial data. Furthermore, if the noise is not in a divergence form, we show that the solution has a finite time blowup (with nonzero probability) for any nonzero initial data. The results on the continuous dependence of solutions on the small random perturbations, alongside with the existence of local strong solutions, are also derived in this work.  
\end{abstract}

\maketitle


\section{Introduction}\label{sec:intro}

In this paper we consider the stochastic aggregation-diffusion equation of Keller--Segel type in $\R^2$ for $t \geq 0$, given by
	\begin{equation}\label{eqn:main_eqn}
				\left\{ \begin{aligned} 
													& \dd \rho  \,=\, \left(\frac{a^2}{2}\Delta\rho - \chi \nabla \cdot(\rho \nabla c)\right) \dd t + \Phi (\rho, \nabla \rho) \dd W(t,x),  \\
													& -\Delta c = \rho, \\
													& \rho(0)=\rho_0,
				\end{aligned}\right.
	\end{equation}
Here $\Phi(\cdot, \cdot)$ is Lipschitz in both arguments,
\begin{equation}\label{eqn:noise}
W(x,s) \defeq \sum_{k=1}^N \alpha_k e_k(x) W_k(s),
\end{equation}
with $\|e_k\|_\infty \leq 1$ and $W_k$ denoting independent scalar Wiener processes. In case $N=\infty$ we additionally assume that $\sum_{k=1}^\infty \alpha_k^2 < \infty$ and $\{e_k \colon k\geq 1\}$ is an orthonormal basis in $L^2(\R^2)$. The initial condition $\rho_0$ is a non-negative integrable function. The solution of the second equation can be expressed as
	\[
		c = (-\Delta)^{-1}\rho = G*\rho(x),
	\]
where $G$ is the Bessel potential; hence we will refer to $\nabla \cdot (\rho \nabla G*\rho)$ as the {\it nonlocal term}, and \eqref{eqn:main_eqn} can be equivalently written in {\it nonlocal form}:
\begin{equation}\label{eqn:main_eqn1}
				\left\{ \begin{aligned} 
													& \dd \rho  \,=\, \left(\frac{a^2}{2}\Delta\rho - \chi \nabla \cdot(\rho \nabla G*\rho )\right) \dd t + \Phi (\rho, \nabla \rho) \dd W(t,x),  \\													& \rho(0)=\rho_0.
				\end{aligned}\right.
	\end{equation}

The deterministic version of the equation \eqref{eqn:main_eqn} is the well-known Keller--Segel model of chemotaxis, originally proposed in \cite{KellerSegel}. An important aspect of the Keller--Segel equation is that the diffusion term $\Delta \rho$ and the nonlocal quadratic aggregation term $\nabla \cdot (\rho \nabla G*\rho)$ are in direct competition: the diffusion term contributes to global existence (in the extreme case it is a heat equation with $\chi = 0$ and $\Phi=0$), while the quadratic term initiates finite time blowup in certain norms. Thus, depending on the mass of the initial value, solutions of the Keller--Segel equation may exist globally (for small initial mass) or blow up in finite time (for large initial mass) (see \cite{Bil,BlanchetCarrilloMasmoudi,BlanchetDolbeaultPerthame,HillenPainter,JagLuc} as well as \cite{Hor1,Hor2} for a compherensive survey). In this work we show that the noise term typically contributes to blowup of solutions, and in certain cases initiates blowup even for small masses.

The Keller--Segel equation can be derived from a system of interacting particles $\{(X_i)(t)\}_{i=1}^N$ satisfying stochastic differential equations
	\[
		\dd X_i(t) = \frac{1}{N-1} \sum_{i\neq j}^N \nabla K(X_i(t) - X_j(t))\dd t + \sqrt{2}\dd B_i(t), \qquad i=1,\ldots,N, \ t>0,
	\]
where $K$ is a radially increasing interaction kernel (such as $K(x)=-|x|^{2-d}$) and $\{B_i(t)\}_{i=1}^N$ are independent Wiener processes. The rigorous derivation of the Keller--Segel equation from the microscopic particle system follows, for example, through the propagation of chaos as $N\to \infty$ (see \cite{CarrilloChoiHauray,JabWan} and references therein).

The particle system which leads to the classical deterministic Keller--Segel only contains idiosyncratic noises that are independent from one particle to another. As the effects of these noises average out in the large particle limit $N\to\infty$ they lead to a deterministic equation. However, incorporating common environmental noises, such as temperature, light, and sound, is crucial for a more precise modeling of chemotaxis.

In \cite{CogFla} the authors obtain a propagation of chaos result for the particle system with common noise  given by
	\[
		\dd X_i(t) = \frac{1}{N} \sum_{i \neq j}^N \nabla K(X_i(t)-X_j(t))\dd t + \sum_{j=1}^\infty \sigma_j(X_i(t))\dd B_j(t), \qquad i=1,\ldots,N,
	\]
where $K$ is an interaction kernel such that $\nabla K$ is uniformly Lipschitz, $\sigma_j\colon \R^d \to \R^d$ are also uniformly Lipschitz and $B_j(t)$ are independent real-valued Brownian motions for $j\in\N$. They show that, the weak limit, in probability, of the empirical measure $\frac{1}{N}\sum_{i=1}^N \delta_{X_i(t)}$ is the measure $\mu(t)$ which satisfies the stochastic PDE
	\[
		\dd \mu(t) = \frac{1}{2} \Delta \mu(t) + \nabla \cdot \big( \mu(t) F(\mu(t)) \big)\dd t - \sum_{j=1}^\infty \nabla \big(\sigma_j(x)\mu(t)\big)\dd B_j(t)
	\]
in some distributional sense, where $F(\mu(t)) = \int_{\R^d} \nabla K(x-y)\dd \mu_y(t)$.

The works addressing stochastic Keller--Segel model include \cite{Fla} and \cite{HuaQiu}. In \cite{Fla}, the authors establish the local existence of solutions for a large class of nonlinear SPDEs, including stochastic Keller--Segel type equation on torus. The work \cite{HuaQiu} by Huang and Qiu addresses a Keller--Segel type stochastic aggregation-diffusion equation in the entire space, in which the interactions are determined by Bessel potentials. The authors obtain the SPDE as the large particle limit of a system of stochastic differential equations with common noise (in divergence form) and prove, using the contraction argument in spirit of \cite{CogFla}, the local existence and uniqueness for weak solutions on the interval $[0,T]$, where $T$ depends on the  $L^4$-norm of the initial value. Furthermore, the authors show that if the $L^4$-norm of the initial value is sufficiently small, the equation admits a global weak solution. 

On the contrary, if the initial data is large, in the deterministic case it is well known (see e.g. \cite{BilKar,BlanchetDolbeaultPerthame}) that global existence of solutions is impossible, since the solution has finite time blowup. The finite time blowup of the solutions are not considered in \cite{HuaQiu}. Hence, to the best of our knowledge, the question of a finite time blowup for stochastic Keller--Segel model, has not been addressed in the literature yet.  Only recently, in \cite{RosSta}, Rosenzweig and Staffilani study a large class of active scalar equations in arbitrary dimension (including parabolic-elliptic Patlak--Keller--Segel equation) and show that a random diffusion term can lead to global solutions for such equations which, without any diffusion, would have finite time blowup.

In this paper we deal with the stochastic Keller--Segel equation, with the stochastic perturbation of the following two types. The first choice of noise is of {\it divergence type}, which is a particular case of the noise in \cite{HuaQiu}:
 \begin{equation}\label{eqn:main_eqn_type1}
				\left\{ \begin{aligned} 
													& \dd \rho  \,=\, \left(\frac{\nu^2 + \sigma^2}{2}\Delta\rho - \chi \nabla \cdot(\rho \nabla G*\rho )\right) \dd t + \sigma  \nabla \rho \dd W(t),  x \in \R^2 \\													& \rho(0)=\rho_0.
				\end{aligned}\right.
	\end{equation}
	where $\sigma>0$ is a constant, $W(t) = (W_1(t), W_2(t))$, with $W_1$ and $W_2$ being independent scalar Wiener processes, and for the convenience of further presentation, we denote 
	$$\nu^2:= a^2 - \sigma^2 > 0.$$
The second choice is a noise of non-divergence type:
\begin{equation}\label{eqn:main_eqn_type2}
				\left\{ \begin{aligned} 
													& \dd \rho  \,=\, \left(\frac{a^2}{2}\Delta\rho - \chi \nabla \cdot(\rho \nabla G*\rho )\right) \dd t + \Phi(\rho) \dd W(t,x),  \\													& \rho(0)=\rho_0,
				\end{aligned}\right.
	\end{equation}
where $\Phi(\cdot)$ is Lipschitz and $W(x,s)$ is defined in (\ref{eqn:noise}). Sometimes we will refer to the noise in (\ref{eqn:main_eqn_type2}) as the general noise. The noise of this type is consistent with the one typically used in the abstract theory of SPDEs by Da Prato and Zabczyk \cite{DapZab92, DapZab96} when studying the long time behavior and invariant measures for stochastic PDEs, both in local \cite{MisStaYip, MisStaYip2, MisStaYip3} and nonlocal  \cite{MisStaHie} settings. 

\bigskip

Let us overview the main results of the present paper.  
\medskip
\begin{itemize} \setlength\itemsep{1em}
\item In Section \ref{sec:existence}, which is devoted to the existence of solutions, we start our analysis with Theorem \ref{thm:exist_loc_spec}. First we show that the equation  (\ref{eqn:main_eqn_type1}) has a local-in-time weak solution, which is in the space $L^p(\R^2), p \geq 2$ with probability 1. 

\medskip

\noindent We further show in Theorem \ref{th:global}, by means of a priori estimates, that the local solution is in fact global, provided the initial mass is sufficiently small. Similar results were obtained in \cite{HuaQiu}  for $p=4$ only.  

\medskip

\noindent We conclude Section \ref{sec:existence} with the analysis of Keller--Segel with general noise (\ref{eqn:main_eqn_type2}). Using the approach of Debussche, Glatt-Holtz, and Temam \cite{DebGHTem}, in Theorem \ref{thm:strong} we show that this equation has a local strong solution for any initial data. 

\item In Section \ref{sec:blowup} we analyze several scenarios, under which the solutions of (\ref{eqn:main_eqn1}) seize to exist globally in time, and define several types of blowup behavior (see Definition \ref{def:bu_types} below). We start with Theorem \ref{thm:bu_special}, in which  we show that for sufficiently large initial mass, the solutions of (\ref{eqn:main_eqn_type1}) have finite time blowup with probability 1.

\medskip

\noindent Theorem \ref{thm:bu_general} establishes a similar result for the solutions of equation (\ref{eqn:main_eqn_type2}). However, the analysis in this case is more difficult, since the $L^1$ norm of solutions is preserved only on average (in contrast with  (\ref{eqn:main_eqn_type1}), where it is preserved almost surely), and there are no a priori bounds on the length of the existence intervals for solutions. 

\medskip

\noindent We conclude Section \ref{sec:blowup} by showing that it is not possible to establish the analog of Theorem \ref{th:global} for the equation (\ref{eqn:main_eqn_type2}) with general noise. In particular, in Theorem \ref{thm:bu_general} we argue that for any nonzero initial data the solution has a finite time blowup, with a nonzero probability. 

\item Finally, in Section \ref{sec:small_perturb} we show, that small random perturbations of the equation (\ref{eqn:main_eqn_type1}) converge to the solutions of deterministic Keller--Segel equation. 
\end{itemize}

\bigskip
\section{Preliminaries}\label{sec:prelim}

We start by noting some properties of the Bessel potential which accounts for the nonlocal contribution. The space of Bessel potentials is defined as the space of function $f$ such that $(-\Delta)^{k/2}f \in L^p(\R^n)$, and it is equivalent to the Sobolev space $W^{k,p}(\R^n)$. Since $c = (-\Delta)^{-1}\rho$, if $\rho \in L^p(\R^n)$ with $1<p<n$, then $c\in W^{2,p}(\R^n)$.

Using the Sobolev embedding $W^{k,p}(\R^n) \subset C^{r,\alpha}(\R^n)$ where $n<pk$, $\alpha \in (0,1]$, $\frac{1}{p}-\frac{k}{n} = -\frac{r+\alpha}{n}$, with the choices of $n=2$, $k=2$, and $r=1$, we see that $\alpha>0$ is equivalent to assuming that $p>2$. Therefore $W^{2,p}(\R^2) \subset C^{1,\alpha}(\R^2)$ for all  $p>2$ and $\alpha>0$.

In particular, we have that
	\begin{equation} \label{eqn:bessel_est}
		\| G * \rho \|_{W^{1,\infty}(\R^2)} \leq C_p \| G*\rho \|_{W^{2,p}(\R^2)} \leq C_p \|\rho \|_p
	\end{equation}
for all $p>2$. Here and in the rest of the paper $\| \, \cdot \, \|_{p}$ denotes the $L^p(\R^n)$-norm.

Let $(\Omega, \mathcal{F},   \mathcal{F}_t, P)$ be a complete filtered probability space.  Denote by $\mathcal{F}^W_t, t \geq 0$ the augmented filtration, generated by the Wiener process in (\ref{eqn:main_eqn}). Following \cite{HuaQiu}, we introduce the spaces:

\begin{definition} 
For each Banach space $(X, \|\cdot\|_X)$, for all $q \geq 1$ and $0 \leq t < \tau \leq T$, denote $S_{\mathcal{F}}^{q}([t, \tau];X)$ to be the set of $X$-valued, $\mathcal{F}_t$ - adapted and continuous processes  $\{\xi_s, s \in [t,\tau]\}$, such that
\[
\|\xi\|_{S_{\mathcal{F}}^q([t,\tau]; X)}:=
\begin{cases}
\left(\E \sup_{s \in [t,\tau]} \|\xi_s\|_{X}^q\right)^{1/q}, \quad q \in [1,\infty);\\
\esssup_{\omega \in \Omega} \sup_{s \in [t, \tau]} \|\xi_s\|_{X}, \quad q = \infty.
\end{cases}
\]
\end{definition}

\begin{definition} 
Denote $L_{\mathcal{F}}^{q}([t, \tau];X)$ to be the set of $X$-valued, predictable processes  $\{\xi_s, s \in [t,\tau]\}$, such that
\[
\|\xi\|_{L_{\mathcal{F}}^q([t,\tau]; X)}:=
\begin{cases}
\left(\E \int_t^\tau \|\xi_s\|_{X}^q \dd s \right)^{1/q}, \quad q \in [1,\infty);\\
\esssup_{\omega \in \Omega} \sup_{s \in [t, \tau]} \|\xi_s\|_{X}, \quad q = \infty.
\end{cases}
\]
\end{definition}

\begin{definition}\label{def:weak}
A family of random functions $\rho(t)$ in $S_{\mathcal{F_W}}^{\infty}([0, T];L^1(\R^2) \cap L^p(\R^2))$ is a {\it weak solution} of (\ref{eqn:main_eqn1}) on $[0,T]$ if for any $\varphi \in C_c^2(\R^2)$ and $t \in [0,T]$ almost surely we have
\begin{align*} 
		\int_{\R^2} & \varphi(x)  \rho(x,t) \dd x - \int_{\R^2} \varphi(x) \rho(x,0)\dd x  =  \\
		& \int_0^t \left( \frac{a^2}{2}\int_{\R^2} \Delta \varphi(x)\rho(x,s) \dd x  - \chi  \int_{\R^2} \rho(s) \nabla  \varphi(x) \nabla_x (\mathcal{G} * \rho(s)) \dd x \right) \dd s \\
&  + \int_0^t (\varphi(\cdot), \Phi \dd W(s,\cdot)),
	\end{align*}
\end{definition}
where
\[
(\varphi(\cdot), \Phi \dd W(s,\cdot)) := \sum_{k=1}^{\infty} a_k \int_{\R^2} \varphi(x)  \Phi(\rho, \nabla \rho) e_k(x) \dd x \dd W_k(s).
\]
The rigorous definition of a {\it strong solution} will we given in Subsection \ref{subsec:loc_exist_gen_noise}.

The following proposition is crucial in establishing the physical properties of the solution of Keller--Segel equation, namely, the preservation of mass and non-negativity. 
\begin{proposition}\label{prop:max_pr}
Suppose $\rho_0 \in L^1(\R^2)$ is such that $\rho_0 \geq 0$ and 
\[
\int_{\R^2} \rho_0(x) \dd x = m_0.
\]
Then a local solution satisfies the following properties.
\begin{enumerate}[label={\bf (\roman{enumi})}] \itemsep=0.8em
\item (Positivity property) If $\rho(x,t)$ is a local solution of either  (\ref{eqn:main_eqn_type1}) or (\ref{eqn:main_eqn_type2}), then for all $t$ from the interval of existence we have $\rho(x,t) \geq 0$ a.s.

\item (Preservation of mass property)
If $\rho(x,t)$  is a local solution of (\ref{eqn:main_eqn_type1}) , then for all $t$ from its interval of existence we have 
$$\int_{\R^2} \rho(x,t) \dd x = m_0 \text{    a.s. }$$  
\item (Preservation of average mass property)  If $\rho(x,t)$ is a local solution of (\ref{eqn:main_eqn_type2}), then for all $t$ from its interval of existence 
$$\E \int_{\R^2} \rho(x,t) \dd x = m_0.$$
\end{enumerate}
\end{proposition}

We postpone the proof of this proposition to subsection \ref{subsec:local_exist}. The proof of the next proposition relies on the generalized Ito's formula \cite{KryIto}. 

\begin{proposition}\label{prop:Ito}
Let $\rho(t,x)$ be a weak solution of (\ref{eqn:main_eqn_type1}). Then
\begin{multline}\label{eqn:Ito_formula}
		\| \rho \|_p^p - \| \rho_0 \|_p^p \\ = - \int_0^t \int_{\R^2} \frac{1}{2} \nu^2 p (p-1) \rho^{p-2} |\nabla \rho|^2 \dd x \dd s 
			- \chi \int_0^t \int_{\R^2} p \rho^{p-1} \nabla \cdot (\rho \nabla c)  \dd x \dd s.
	\end{multline}
\end{proposition}
\begin{proof}
Note that, formally, for $p>2$ and $\mathcal{F}(\rho)=\rho^p$, we have
	\[
		\dd\mathcal{F} = p \rho^{p-1} \dd \rho + \frac{1}{2} p (p-1) \rho^{p-2} \sigma^2|\nabla \rho|^2 \dd t. 
	\]
Hence, using \cite{KryIto}, Theorem 2.1, we have
	\begin{multline}\nonumber
		\| \rho \|_p^p - \| \rho_0 \|_p^p = -\int_0^t \int_{\R^2} \left( \frac{1}{2}(\nu^2 + \sigma^2) p (p-1) \rho^{p-2} | \nabla \rho|^2 + \chi p \rho^{p-1} \nabla \cdot (\rho \nabla c) \right)\dd x \dd s \\ 
			+ \int_0^t \int_{\R^2} \frac{1}{2} p (p-1) \rho^{p-2} \sigma^2 |\nabla \rho|^2 \dd x \dd s - \int_0^t \int_{\R^2} \sigma p \rho^{p-1} \nabla \rho  \dd W(t).
	\end{multline}
Since $$\int_{R^2} p \rho^{p-1}\nabla \rho \,\dd x = \int_{\R^2} \nabla (\rho^p)\dd x=0,$$ 
the formula (\ref{eqn:Ito_formula}) follows.  
\end{proof}

\bigskip
\section{Existence of solutions.}\label{sec:existence}
In this section we show that both equations (\ref{eqn:main_eqn_type1}) and (\ref{eqn:main_eqn_type2}) have local solutions. In addition, we show that for small initial data, the equation  (\ref{eqn:main_eqn_type1})  has a global in time solution. For the rest of the paper we denote by $H^1$ the Sobolev space $W^{1,2}$.
\subsection{Local existence of weak solutions for equation (\ref{eqn:main_eqn_type1}).}\label{subsec:local_exist}
In this section we show that the stochastic Keller--Segel equation with the noise in divergence form (\ref{eqn:main_eqn_type1}) is locally well posed.
\begin{theorem}\label{thm:exist_loc_spec}
For any $\rho_0 \in L^1(\R^2) \cap L^p(\R^2) \cap H^{1}(\R^2)$, $p \geq 2$,  with $\int_{\R^2} \rho_0(x) \dd x = m_0$, there exists $T = T(m_0)>0$ such that (\ref{eqn:main_eqn_type1}) has a weak solution $\rho(x,t) \in L^2_{\mathcal{F_W}}(0,T,H^1(\R^2)) \cap S_{\mathcal{F_W}}^{\infty}([0, T];L^1(\R^2) \cap L^p(\R^2)).$ 
\end{theorem}
\begin{proof} 
For any $p \geq 2$ and $T>0$ introduce the space
	\[
		S_T \defeq \Bigl\{ \rho=\rho(x,t,\omega) \in S_{\mathcal{F_W}}^{\infty}([0, T];L^1(\R^2) \cap L^p(\R^2)) \colon \esssup_{\omega \in \Omega} \sup_{t\in [0,T]} \| \rho \|_p < \infty. \Bigr\}
	\]
For $\xi \in S_T$, we define the linearized equation
	\begin{equation}\label{eqn:main_linear}
				\left\{ \begin{aligned} 
													& \dd \rho  \,=\, \left(\frac{1}{2}(\nu^2 + \sigma^2)\Delta\rho - \chi \nabla \cdot(\rho \nabla G * \xi)\right) \dd t + \sigma \nabla \rho \dd W(t),  \\
													& \rho(0)=\rho_0.
				\end{aligned}\right.
	\end{equation}
It follows from \cite{KryMaxPr} (Theorem 5.1 and Theorem 7.1) that this equation has a unique solution $\rho_\xi \in L^2_{\mathcal{F_W}}(0,T,H^1(\R^2)) \cap S_{\mathcal{F_W}}^{\infty}([0, T];L^1(\R^2) \cap L^p(\R^2))$. 
Now, we define the metric space
	\[
		B \defeq \Bigl\{ \xi \in S_T \colon \esssup_{\omega\in\Omega} \sup_{t\in[0,T]} \| \xi \|_p \leq 2\|\rho_0\|_p \Bigr\},
	\]
with the usual metric given by
	\[
		d(u,v) \defeq \esssup_{\omega\in\Omega} \sup_{t\in[0,T]} \| u-v\|_p.
	\]
In order to obtain local existence, our goal is to show that 
	\begin{equation}\label{eqn:map}
	B \ni \xi \mapsto \rho_\xi \in B,
	\end{equation}
	as well as the contraction property
	\begin{equation}\label{eqn:contraction}
	d(\rho_\xi, \rho_{\eta}) \leq z  d(\xi,{\eta}), \ \ \text{ for } \xi,{\eta} \in B \text{ and } 0<z<1.
	\end{equation}
To show (\ref{eqn:map}), let	$\rho_0 \in L^1(\R^2) \cap L^p(\R^2)$, $\rho_0 \geq 0$.  In a similar way to Proposition \ref{prop:Ito}, we may apply generalized Ito's formula for the solution $\rho(x,t)$ of the linearized equation (\ref{eqn:main_linear}). This way, for any $\xi \in B$, we have 
\begin{multline}\label{eqn:Ito_formula_lin}
		\| \rho(t) \|_p^p - \| \rho_0 \|_p^p \\ = - \int_0^t \int_{\R^2} \frac{1}{2} \nu^2 p (p-1) \rho^{p-2} |\nabla \rho|^2 \dd x \dd s 
			- \chi \int_0^t \int_{\R^2} p \rho^{p-1} \nabla \cdot (\rho \nabla(G*\xi))  \dd x \dd s.
	\end{multline}
Furthermore, for any $q \geq 2$ we may estimate the last term as follows:
\begin{align} \label{eqn:est_nonloc}
\nonumber		p \left| \int_{\R^2} \rho^{p-1} \nabla \cdot (\rho \nabla (G * \xi)) \dd x \right| &= p(p-1) \left| \int_{\R^2} \rho \rho^{p-2} \nabla \rho \cdot \nabla (G * \xi) \dd x \right| \\
																	\nonumber									  &\leq p(p-1) \int_{\R^2} \big| \rho \rho^{(p-2)/2} \rho^{(p-2)/2} \nabla \rho \cdot \nabla (G * \xi) \big| \dd x \\
																	\nonumber									 \text{(by H\"{o}lder inequality)} & \leq p(p-1) \| \rho^{(p-2)/2}\nabla \rho \|_2 \|\rho^{1+(p-2)/2} \nabla (G * \xi)\|_2 \\
																										 \text{(by Young inequality)} & \leq 2p(p-1)\varepsilon \| \rho^{(p-2)/2}\nabla \rho \|_2^2 + \frac{p(p-1)}{2\varepsilon} \| \rho^{p/2} \nabla (G * \xi) \|_2^2 \\
																	\nonumber									  & \leq 2p(p-1)\varepsilon \| \rho^{(p-2)/2}\nabla \rho \|_2^2 + \frac{p(p-1)}{2\varepsilon} \| \nabla (G * \xi) \|_\infty^2 \|\rho^{p/2}\|_2^2 \\
																	\nonumber									  \text{(by estimate (\ref{eqn:bessel_est}))} & \leq 2p(p-1)\varepsilon \| \rho^{(p-2)/2}\nabla \rho \|_2^2 + C_{q} \frac{p(p-1)}{2\varepsilon}\| \xi \|_q^2 \|\rho\|_p^p.
	\end{align}

Choosing $\varepsilon>0$ such that $2\chi \varepsilon p (p-1) \leq \frac{1}{2} \nu^2 p(p-1)$ yields
	\[
		\| \rho(t) \|_p^p \leq \| \rho_0 \|_p^p + \frac{1}{2\varepsilon} \int_0^t \| \xi \|_p^2 \|\rho\|_p^p \dd s.
	\]
Then, by Gr\"{o}nwall's inequality,
	\begin{equation} \label{eqn:contraction2}
		\| \rho(t) \|_p^p \leq \| \rho_0 \|_p^p \exp \left( \int_0^t \| \xi \|_p^2 \dd s \right) \leq \| \rho_0 \|_p^p \exp \left( 4 \|\rho_0\|_p^2 T \right) \leq 2^p \|\rho_0 \|_p^p
	\end{equation}
for $T$ small enough. Therefore $\rho_\xi \colon B \to B$.

We now proceed with the proof of the contraction property (\ref{eqn:contraction}). Let $\xi, \eta \in B$, and let $\rho_\xi$ and $\rho_\eta$ the corresponding solutions of \eqref{eqn:main_linear} with the same initial data $\rho_0$. Define $\bar{\rho} \defeq \rho_\xi - \rho_\eta$. Then $\bar{\rho}$ satisfies
	\[
		\dd \bar{\rho} = \left( \frac{\nu^2 + \sigma^2}{2} \Delta \bar{\rho} - \chi \nabla \big( \rho_\xi \nabla(G*\xi) - \rho_\eta \nabla(G*\eta) \big)\right) \dd t + \sigma\, \nabla \bar{\rho} \dd W(t).
	\]
Again, using Ito's formula in a similar way to \eqref{eqn:Ito_formula}, we get
	\begin{multline} \nonumber
		\| \bar{\rho} \|_p^p = -\int_0^t \int_{\R^2} \frac{1}{2} \nu^2 p(p-1) \bar{\rho}^{p-2} |\nabla\bar{\rho}|^2 \dd x \dd s  \\ - \chi \int_0^t \int_{\R^2} p \bar{\rho}^{p-1} \nabla \cdot \big( \rho_\xi \nabla (G*\xi) - \rho_\eta \nabla (G * \eta)  \big) \dd x \dd s.
	\end{multline}
Now we integrate the second term by parts, and add and subtract $\rho_\eta \nabla(G *\xi)$,
	\begin{align*}
		-\int_{\R^2} \bar{\rho}^{p-1} &\nabla \cdot \big( \rho_\xi \nabla(G*\xi) - \rho_\eta \nabla(G*\eta)  \big) \dd x \\
		&= \int_{\R^2}  (p-1) \bar{\rho}^{p-2} \nabla\bar{\rho}\big( \rho_\xi \nabla(G*\xi) -\rho_\eta \nabla(G*\xi) + \rho_\eta \nabla(G*\xi) - \rho_\eta \nabla (G*\eta)  \big)\dd x \\
		&= (p-1) \int_{\R^2} \bar{\rho}^{p-1} \nabla\bar{\rho} \cdot \nabla (G*\xi) \dd x  \\
		&\qquad\qquad\qquad\qquad + (p-1) \int_{\R^2} \bar{\rho}^{p-2} \nabla \bar{\rho} \cdot \big( \rho_\eta \nabla(G*(\xi-\eta)) \big) \dd x  \\
		&=: I_1 + I_2.
	\end{align*}
	
Observe that $\bar{\rho}^{p-1} \nabla\bar{\rho} = \bar{\rho}^{p/2}\bar{\rho}^{p/2 - 1} \nabla \bar{\rho} = \frac{2}{p} \bar{\rho}^{p/2} \nabla \big( \bar{\rho}^{p/2}\big)$ as well as $\| \nabla (G * \xi) \|_\infty \leq C_p \| \xi \|_p$ by (\ref{eqn:bessel_est}). Then we estimate
	\begin{align*}
		\int_0^t |I_1(s)|\dd s &\leq \frac{2(p-1)}{p} \int_0^t \left( \int_{\R^2} \left|\nabla \bar{\rho}^{p/2} \right|^2  \dd x \right)^{1/2} \left( \int_{\R^2} \bar{\rho}^p |\nabla (G * \xi)|^2   \right)^{1/2} \dd s \\
									 &\leq \frac{4(p-1)\varepsilon_1}{p} \int_0^t \int_{\R^2} \left|\nabla \bar{\rho}^{p/2} \right|^2 \dd x \dd s + C\frac{(p-1)}{p\varepsilon_1} \int_0^t \int_{\R^2} \bar{\rho}^p \|\xi\|_p^2 \dd x \dd s,
	\end{align*}
where the first inequality follows from H\"{o}lder and the second one from Young inequalities.

We now estimate $I_2$. Again, using H\"{o}lder and Young inequalities, 
	\begin{align*}
		|I_2| &\leq \frac{2(p-1)}{p} \left( \int_{\R^2} \left|\nabla \bar{\rho}^{p/2} \right|^2  \dd x \right)^{1/2} \left( \int_{\R^2}\rho_\eta^2\bar{\rho}^{p-2} |\nabla(G * (\xi-\eta))|^2  \dd x \right)^{1/2} \\
			  & \leq \frac{4(p-1)\varepsilon_2}{p}  \int_{\R^2} \left|\nabla \bar{\rho}^{p/2} \right|^2  \dd x + \frac{p-1}{p\varepsilon_2}  \int_{\R^2} \bar{\rho}^{p-2} \rho_\eta^2 |\nabla(G * (\xi-\eta))|^2 \dd x.
	\end{align*}
Applying the Young's inequality $ab \leq \frac{p-2}{p}a^{p/(p-2)}+\frac{2}{p}b^{p/2}$ to the second term on the right-hand side, we get
	\begin{align*}
		\int_{\R^2} \bar{\rho}^{p-2} \rho_\eta^2 |\nabla (G * (\xi-\eta))|^2 \dd x &\leq \frac{p-2}{p} \int_{\R^2} \bar{\rho}^p\dd x + \frac{2}{p} \int_{\R^2} \rho_\eta^p |\nabla(G * (\xi-\eta))|^p \dd x \\ 
																											 &\leq \frac{p-2}{p} \int_{\R^2} \bar{\rho}^p \dd x + \frac{2}{p} C \|\xi-\eta\|_p^p  \int_{\R^2} \rho_\eta^p \dd x \\
																											 &\xupref{eqn:contraction2}{\leq} \frac{p-2}{p} \int_{\R^2} \bar{\rho}^p \dd x + \frac{2^{p+1}}{p} C \|\rho_0\|_p^p \, \| \xi - \eta \|_p^p.
	\end{align*}
	
Altogether,
	\begin{align*}
		\| \bar{\rho} \|_p^p &\leq -\int_0^t \int_{\R^2} \frac{2(p-1)\nu^2}{p}  \left| \nabla \bar{\rho}^{p/2} \right|^2 \dd x \dd s + \frac{4(p-1)(\varepsilon_1+\varepsilon_2)}{p} \int_0^t \int_{\R^2}  \left| \nabla \bar{\rho}^{p/2} \right|^2 \dd x \dd s \\
									&\qquad + \frac{p-1}{p\varepsilon_1} \int_0^t  C\| \xi\|_p^2 \left( \int_{\R^2}\bar{\rho}^p \dd x \right)\dd s + \frac{(p-1)(p-2)}{p^2 \varepsilon_2} \int_0^t \int_{\R^2} \bar{\rho}^p \dd x \dd s \\
								    	&\qquad\qquad + \frac{(p-1)2^{p+1}}{p^2 \varepsilon_2} C \|\rho_0\|_p^p \int_0^t \| \xi - \eta \|_p^p \dd s.
	\end{align*}
Choosing $\varepsilon_1$ and $\varepsilon_2$ so that $\nu^2 = 2(\varepsilon_1 + \varepsilon_2)$ and noting that $\esssup_{\omega\in\Omega} \sup_{t\in[0,T]} \| \xi \|_p^2 \leq 4 \| \rho_0 \|_p^2$ as $\xi \in B$, we have
	\begin{equation}\label{eqn:difference}
		\| \bar{\rho} \|_p^p \leq \frac{p-1}{p} \left( \frac{4 C_p\| \rho_0\|_p^2}{\varepsilon_1} + \frac{p-2}{p \varepsilon_2}\right) \int_0^t \| \bar{\rho} \|_p^p \dd s + \frac{(p-1)2^{p+1}}{p^2 \varepsilon_2} C_p \|\rho_0\|_p^p \int_0^t \| \xi - \eta \|_p^p \dd s.
	\end{equation}
Using Gronwall's inequality once again, the estimate (\ref{eqn:difference}) yields $\| \bar{\rho} \|_p^p \leq C_1 e^{C_2 t} \int_0^t \| \xi - \eta \|_p^p \dd s$ for some constants $C_1>0$ and $C_2>0.$ This means that $d(\rho_\xi,\rho_\eta) \leq C_1 t e^{C_2 t} d(\xi,\eta)$, i.e., we have a contraction property (\ref{eqn:contraction}) for sufficiently small $T>0$.
\end{proof}	

\bigskip

We now prove Proposition \ref{prop:max_pr}. 

\begin{proof}[Proof of Proposition \ref{prop:max_pr}]
It follows from \cite[Theorem 5.12]{KryMaxPr}, that the solution $\rho$ of the linearized equation (\ref{eqn:main_linear}) satisfies $\rho(x,t) \geq 0$ if $\rho_0(x) \geq 0$, for every $\xi \in S_T$. In particular, this statement holds if $\xi = \rho(x,t)$, which is the fixed point in the proof of Theorem \ref{thm:exist_loc_spec}. The proof of {\bf(i)} for the general noise case (\ref{eqn:main_eqn_type2}) is analogous. 

To prove {\bf (ii)}, let us write the solution of (\ref{eqn:main_eqn_type1}) in the mild form
\begin{equation}\label{eqn:mild1}
\rho(t) = S(t) \rho_0 -  \chi \int_{0}^t S(t-s)\big[\nabla \cdot(\rho(s) \nabla (G*\rho(s)))\big] \dd s + \sigma \int_{0}^t S(t-s)[\nabla \rho(s)] \dd W(s)
\end{equation}
where $S(t)$ is the heat semigroup. Using the properties of convolution (see e.g. \cite{BilKar})
\begin{equation}\label{prop1}
\int_{\R^2} S(t) v(x) \dd x = \int_{\R^2} v(x) \dd x,
\end{equation}
and
\begin{equation}\label{prop2}
\int_{\R^2} \nabla S(t) v(x) \dd x = 0 \ \text{ for every } v \in L^1(\R^2),
\end{equation}
we obtain
\[
\int_{\R^2} \rho(x,t) \dd x = \int_{\R^2} \rho_0(x) \dd x
\]
from (\ref{eqn:mild1}).

Finally, to show {\bf (iii)}, for fixed $t>0$ from the interval of the existence of the solution, we represent the mild solution of (\ref{eqn:main_eqn_type2}) as

\begin{equation}\label{eqn:mild2}
\rho(t) = S(t) \rho_0 -  \chi \int_{0}^t S(t-s)\big[\nabla \cdot(\rho(s) \nabla (G*\rho(s)))\big] \dd s + \int_{0}^t S(t-s)\big[\Phi(\rho, \nabla \rho) \dd W(x,s)\big].
\end{equation}
Taking the expected values of both sides, (\ref{eqn:mild2}) becomes
\begin{equation}\label{eqn:mild2_1}
\E \rho(t) = S(t) \rho_0 -  \chi \E \int_{0}^t S(t-s)\big[\nabla \cdot(\rho(s) \nabla (G*\rho(s)))\big] \dd s.
\end{equation}
Then using (\ref{prop1}), (\ref{prop2}) and Fubini's theorem we get
\[
\E \int_{\R^2} \rho(t,x) \dd x = \int_{\R^2} \rho_0(x) \dd x,
\]
as stated.
\end{proof}

\bigskip
\subsection{Global existence of solutions for (\ref{eqn:main_eqn_type1})}\label{subsec:global_exist}

In this section we show that for sufficiently small initial masses, the local solution, obtained in the previous section, is in fact global. 

\begin{theorem} \label{th:global}
Let $\int_{R^2} \rho_0 \dd x = m_0$. There exists a constant $C = C(p)$, such that if 
	\begin{equation}\label{eqn:smallness_cond1}
	 C \left( \frac{-\nu^2 p (p-1)}{2} + \chi \right) m_0^{-p/(p-1)} + \chi m_0^{(p-2)/(p-1)} \leq 0 \qquad \text{ for } p \geq 2,
	\end{equation}

then the solution of \eqref{eqn:main_eqn} exists for all $t \geq 0$ and $\| \rho(t) \|_p \leq \| \rho_0 \|_p$ a.s.
\end{theorem}

\begin{proof}  Let us start with the proof of this result for $p=2$, which is technically simpler due to Nash inequality.  By the Sobolev embedding $W^{k,p}(\R^n) \subset C^{r,\alpha}(\R^n)$, $\alpha\in(0,1]$, $\frac{1}{p}-\frac{k}{n} = -\frac{r+\alpha}{n}$, with $p=n=k=2$ and $\alpha=1$, $r=0$, we have 
	\[
		\| \rho \|_{C^{0,1}(\R^2)} \leq C \| \rho \|_{W^{2,2}(\R^2)}.
	\]
Since $\| \nabla u \|_{L^\infty(\R^2)} \leq C \| u \|_{C^{0,1}(\R^2)}$ for any $u\in W^{1,\infty}(\R^2) \cap C^{0,1}(\R^2)$, we get
	\[
		\| \nabla(G * \rho) \|_\infty \leq C \| G * \rho \|_{C^{0,1}(\R^2)} \leq C \| G*\rho \|_{W^{2,2}(\R^2)} \leq C\| \rho \|_2.
	\]
Now, combining the Ito's formula (\ref{eqn:Ito_formula}) and the estimate (\ref{eqn:est_nonloc}) with $p=q=2$, we obtain

	\begin{align*}
		\| \rho \|_2^2 - \| \rho_0\|_2^2 & =  -\nu^2 \int_0^t \int_{\R^2} | \nabla \rho |^2 \dd x \dd s - 2 \chi \int_0^t \int_{\R^2}  \nabla \cdot (\rho \nabla(G*\rho))  \dd x \dd s  \\
&\leq -\nu^2 \int_0^t \int_{\R^2} | \nabla \rho |^2 \dd x \dd s +  \chi \varepsilon \int_0^t \| \nabla \rho \|_2^2 \dd s + \chi C \int_0^t \| \rho \|_q^2 \| \rho \|_2^2 \dd s \\
													&= \big( -\nu^2 + \chi \big) \int_0^t \|\nabla \rho\|_2^2 \dd s +\chi C \int_0^t \| \rho \|_2^4 \dd s.
	\end{align*}
	
Using the Nash inequality $\| u \|_2^{1+2/n} \leq C \| u \|_1^{2/n} \| \nabla u \|_2$ with $n=2$ and $\| \rho \|_1 = m_0$, we have $\| \rho \|_2^2 \leq C m_0 \| \nabla \rho \|_2$. Thus
	\[
		\| \rho \|_2^2 - \| \rho_0\|_2^2 \leq \left( -\nu^2 + \chi  +C \chi m_0^2 \right) \int_0^t \| \nabla \rho \|_2^2 \dd s
	\]
	
The above inequality implies that under (\ref{eqn:smallness_cond}) we have $\| \rho(t) \|_2^2 \leq \| \rho_0 \|_2^2$ for all $t \geq 0$, yielding the global existence of a weak solution for $p=2$.

\medskip

We now proceed with the proof of Theorem \ref{th:global} for $p>2$.  Once again, combining (\ref{eqn:Ito_formula}) and (\ref{eqn:est_nonloc}) with $p>2$ and $q>2$, we have
	\[
		\| \rho \|_p^p - \| \rho_0 \|_p^p \leq \left( \frac{-\nu^2 p(p-1)}{2} + \varepsilon \chi \right) \int_0^t \| \nabla \rho^{\frac{p}{2}} \|_2^2 \dd s + \frac{C \chi}{\varepsilon} \int_0^t \|\rho \|_q^2 \| \rho \|_p^p \dd s.
	\]
Next we use the Gagliardo-Nirenberg inequality, $\| D^j u \|_{\tld{p}} \leq C \| D^m u \|_r^\alpha \| u \|_{\tld{q}}^{1-\alpha}$ with $\frac{1}{\tld{p}} = \frac{j}{n} + \left(\frac{1}{r}-\frac{m}{n}\right)\alpha + \frac{1-\alpha}{\tld{q}}$. We take $j=0$, $r=2$, $m=1$, $n=2$, $\tld{p}=s \geq 1$ so that $\tld{q} = (1-\alpha) s \geq 1$, and apply the inequality to $u=\rho^{p/2}$:
	\[
		\| \rho^{p/2} \|_s \leq C \| \nabla \big(\rho^{p/2}\big) \|_2^\alpha \, \| \rho^{p/2} \|_{(1-\alpha)s}^{1-\alpha}.
	\]
	
On the other hand, 
	\[
		\| \rho^{p/2} \|_s = \| \rho \|_{ps/2}^{p/2} \qquad \text{and} \qquad \|\rho^{p/2}\|_{(1-\alpha)s}^{1-\alpha} = \| \rho \|_{p(1-\alpha)s/2}^{p(1-\alpha)/2}.
	\]
So, the Gagliardo-Nirenberg inequality for $\rho^{p/2}$ becomes
	\begin{equation}\label{eqn:GagNir_ineq}
		\| \rho \|_{ps/2}^{p/2} \leq C \| \nabla \big(\rho^{p/2}\big) \|_2^\alpha \, \| \rho \|_{p(1-\alpha)s/2}^{p(1-\alpha)/2}.
	\end{equation}
Next, we apply the interpolation inequality $\| f \|_{\tld{p}}^{\tld{p}} \leq \| f \|_{p_0}^{p_0(1-\theta)} \|f\|_{p_1}^{p_1 \theta}$ with $\tld{p} = (1-\theta)p_0 + \theta p_1$. Taking
	\[
		\tld{p} = \frac{p(1-\alpha)s}{2}, \quad p_0 = \frac{ps}{2}, \quad p_1=1,
	\]
and recalling $\| \rho \|_1=m_0$, we have  $\| \rho \|_{p(1-\alpha)s/2}^{p(1-\alpha)s/2} \leq m_0^\theta \| \rho \|_{ps/2}^{ps(1-\theta)/2}$. So, we get that $ \| \rho \|_{p(1-\alpha)s/2}^{p(1-\alpha)/2} \leq m_0^{\theta/s}\| \rho \|_{ps/2}^{p(1-\theta)/2}.$ Now, by \eqref{eqn:GagNir_ineq},
	\[
		\| \nabla \big( \rho^{p/2} \big) \|_2 \geq m_0^{-\theta/(\alpha s)} \| \rho \|_{ps/2}^{p\theta/(2\alpha)}.
	\]
Also, $\tld{p} = (1-\theta)p_0 + \theta p_1$ implies that $\frac{\theta}{\alpha}= \frac{ps}{ps-2}$. Hence,
	\[
		\| \rho \|_{ps/2}^{p\theta/(2\alpha)} = \left( \int_{\R^2} \rho^{ps/2} \dd x \right)^{p/(ps-2)} \qquad \text{ and } \qquad m_0^{-\theta/(\alpha s)} = m_0^{-p/(ps-2)}.
	\]
Altogether,
	\begin{multline} \nonumber
		\| \rho \|_p^p - \| \rho_0 \|_p^p  \leq C \left( \frac{-\nu^2 p (p-1)}{2} + \varepsilon \chi \right) m_0^{-2p/(ps-2)} \int_0^t \left( \int_{\R^2} \rho^{ps/2} \dd x \right)^{2p/(ps-2)} \dd s \\ + \frac{\chi}{\varepsilon} \int_0^t \| \rho \|_q^2 \| \rho \|_p^p \dd s.
	\end{multline}
	
Now taking $s=2$ and $\varepsilon=1$ the right-hand side becomes
	\begin{multline} \nonumber
		\int_0^t C \left( \frac{-\nu^2 p (p-1)}{2} + \chi \right) m_0^{-p/(p-1)} \left( \int_{\R^2} \rho^{p} \dd x \right)^{p/(p-1)}  + \chi \| \rho \|_p^p \| \rho \|_q^2 \dd s \\
													= \int_0^t \| \rho \|_p^p \left[ C \left( \frac{-\nu^2 p (p-1)}{2} + \chi \right) m_0^{-p/(p-1)} \| \rho \|_p^{p/(p-1)} + \chi \| \rho \|_q^2 \right] \dd s.
	\end{multline}
	
We use the interpolation inequality $\| \rho \|_q^2 \leq \| \rho \|_p^{2(1-\theta)}\| \rho \|_1^{2\theta}$ with $\frac{1}{q}= \frac{1-\theta}{p}+\theta$ to obtain $\| \rho \|_q^2 \leq \| \rho \|_p^{2(1-\theta)} m_0^{2\theta}$. Noting that $2(1-\theta) = \frac{p}{p-1}$ and $2\theta = \frac{p-2}{p-1}$, we obtain that
	\[
		\| \rho \|_p^p - \| \rho_0 \|_p^p \leq \int_0^t \left[ C \left( \frac{-\nu^2 p (p-1)}{2} + \chi \right) m_0^{-p/(p-1)} + \chi m_0^{(p-2)/(p-1)} \right]\| \rho \|_p^p \| \rho \|_{p}^{p/(p-1)} \dd s\leq  0
	\]
for $m_0$ sufficiently small.
\end{proof}

\begin{corollary}
Let $\rho^1(x,t)$ and $\rho^2(x,t)$ be two solutions of (\ref{eqn:main_eqn_type1}) with both $\rho^1(x,0)$ and $\rho^2(x,0)$ satisfying (\ref{eqn:smallness_cond1}).
Then there is $C>0$, independent of $\rho^1$ and $\rho^2$, such that
\begin{equation}\label{eqn:cont_dep}
\E \sup_{t \in [0,T]}\|\rho^1(x,t) - \rho^2(x,t)\|_{p}^p \leq C \|\rho^1(x,0) - \rho^2(x,0)\|_{p}^p 
\end{equation}
\end{corollary}
\begin{proof}
The condition (\ref{eqn:smallness_cond1}) guarantees that both solutions $\rho^1(x,t)$ and $\rho^2(x,t)$ are defined globally and bounded in $L^p(\R^2)$. The proof of this corollary now follows the lines of the proof of the contraction property \eqref{eqn:contraction} in Theorem \ref{thm:exist_loc_spec}.
\end{proof}

\bigskip
\subsection{Local existence of strong solutions for (\ref{eqn:main_eqn_type2}).}\label{subsec:loc_exist_gen_noise}

In this section we will apply an abstract result by Debussche, Glatt-Holtz, and Temam \cite{DebGHTem} to establish the existence of local strong solutions for the equation \eqref{eqn:main_eqn_type2}.  For the convenience of the reader, we introduce the necessary notions and results from \cite{DebGHTem}.

Fix a pair of separable Hilbert spaces $V \subset H$, such that the embedding is compact and dense. This way we have a Gelfand triple $V \subset H \subset V^{'}$, where $V^{'}$ is the dual of $V$, relative to dot products in $H$. Introduce $(\cdot, \cdot)$,  $|\cdot|$, $((\cdot, \cdot))$,  $\|\cdot\|$ to be the norms and inner products in $H$ and $V$ respectively. The duality
product between $V^{'}$ and $V$ is written $\langle \cdot, \cdot \rangle$.

The framework developed in \cite{DebGHTem} deals with abstract stochastic evolution equations of the form
	\begin{equation}\label{eqn:temam}
	 \dd U+(AU+B(U,U)+F(U)) \dd t=\sigma(U)\dd W, \qquad  U(0)=U_0.
	\end{equation}
Here $A \colon D(A) \subset H \rightarrow H$ is an unbounded, densely defined, bijective, linear operator such that $\left(A U, U^\sharp \right)=\left(\left(U, U^{\sharp}\right)\right)$ for all $U, U^{\sharp} \in D(A)$. The operator $B$ is a bilinear form mapping $V \times D(A)$ continuously to $V^{\prime}$ and $D(A) \times D(A)$ continuously to $H$, and the nonlinearity $F:V \to H$ satisfies the Lipschitz condition. Finally, we assume that
\[
\sigma: H \to L^2(\mathcal{U},H),
\]
also satisfies the Lipschitz condition, where $\mathcal{U}$ is an auxiliary Hilbert space and $L^2(\mathcal{U},H)$ is the space of Hilbert-Schmidt operators between $\mathcal{U}$
and $H$.

\begin{definition}[cf. Definition 2.2 in \cite{DebGHTem}]\label{def:strong} 
Let $ S = (\Omega,F, \{F_t\},t\geq 0, P, W)$ be a fixed
stochastic basis and suppose $U_0 \in V$.
 
\begin{enumerate}[label={\bf (\roman{enumi})}] \itemsep=0.8em
\item A pair $(U, \tau)$ is  a {\it local strong solution} of (\ref{eqn:temam}) if $\tau$ is a strictly positive stopping time and $U(\cdot \land \tau)$
is an $F_t$-adapted process in $V$ so that
\begin{equation}\label{eqn:2.18}
U(\cdot \land \tau) \in L^2(\Omega, C([0,\infty); V)) \cap L^2(\Omega, L^2_{loc}([0, \infty); D(A)))
\end{equation}
and
\begin{equation}\label{eqn:2.19}
U(t \land \tau) + \int_{0}^{t \land \tau} (AU + B(U,U) + F(U)) \dd s = U_0 +  \int_{0}^{t \land \tau} \sigma(U) \dd W.
\end{equation}
\item A strong solution of \eqref{eqn:temam} is {\it pathwise unique} up to the stopping time $\tau>0$ if for any
pair of strong solutions $(U^1, \tau)$ and $(U^2, \tau)$ with the same initial condition $U_0$ we have
\[
P\{U^1(t\land \tau) - U^2(t\land \tau) = 0 \text{ for all } t \geq 0\} = 1.
\]
\end{enumerate}
\end{definition}

The following result was shown in \cite[Theorem 2.1]{DebGHTem}:
\begin{theorem}\label{thm:temam}
Suppose the bilinear form $B(U,U)$ satisfies
\begin{equation}
|\langle B (U, U^{\sharp}), U^{\flat}\rangle | \leq c_{0}\|U\| |A U^{\sharp}| \|U^{\flat}\| \quad \text { for all } U, U^{\flat} \in V, \ U^{\sharp} \in D(A) \label{eqn:bilinear_2}
\end{equation}
and
\begin{multline}
\| \langle B(U, U^{\sharp}), U^{\flat}\rangle \| \leq c_{0}\|U\|^{1 / 2} |A U|^{1 / 2} \|U^{\sharp}\|^{1 / 2} |A U^{\sharp}|^{1 / 2} |U^{\flat}| \\ \text { for all } U, U^{\sharp} \in D(A), U^{\flat} \in H. \label{eqn:bilinear_3}
	\end{multline}
Then the equation (\ref{eqn:temam}) has a unique local strong solution, defined for $t\in[0,\zeta(\omega))$, where $\zeta(\omega)$ satisfies
	\[
		\sup_{t\in[0,\zeta]} \|U\|^2 + \int_0^\zeta |AU|^2\dd s = \infty
	\]
almost surely on the set $\{\zeta < \infty \}$.
\end{theorem}

\medskip

The main result of this section is the following theorem.

\begin{theorem}\label{thm:strong}
For any $\rho_0 \in H^1(\R^2)$, the equation (\ref{eqn:main_eqn_type2}) has a unique local strong solution defined for $t\in[0,\tau(\omega))$ in the sense of Definition \ref{def:strong}, where $\tau(\omega)$ satisfies
	\[
		\sup_{t \in [0,\tau]} \|\rho\|_{H^1(\R^2)} + \int_0^\tau \| D^2 \rho \|_2 \dd s = +\infty
	\]
almost surely on the set $\{\tau < \infty \}$.
\end{theorem}
\begin{proof}
 For $u \in V$ and $u^\sharp \in H$ define
 $B(u,u^\sharp) := \nabla \cdot \big( u \nabla (G*u^\sharp) \big) \colon V \times H \to H$ with $D(A)=H^2(\R^2)$, $V=H^1(\R^2)$, $H=L^2(\R^2)$ and the $|\cdot|_{H}$ norm will be denoted with $\|\cdot\|_{H}$.  In order to apply Theorem \ref{thm:temam}, we need to check that the conditions \eqref{eqn:bilinear_2} and \eqref{eqn:bilinear_3} on the bilinear form are satisfied.

Let us start with verifying the condition \eqref{eqn:bilinear_2}. For all $u,u^\flat \in V$ and $u^\sharp \in D(A)$,
	\[
		\langle B(u,u^\sharp),u^\flat \rangle = \int_{\R^2} \nabla \cdot \big( u \nabla (G*u^\sharp) \big) u^\flat \dd x = - \int_{\R^2} u \nabla (G*u^\sharp) \cdot \nabla u^\flat \dd x.
	\]
Hence, by Holder inequality, $\big| \langle B(u,u^\sharp),u^\flat \rangle  \big| \leq \| \nabla u^\flat \|_H \| u\nabla(G*u^\sharp) \|_H$. In view of the embedding $W^{2,2}(\R^2)\subset W^{1,\infty}(\R^2),$ we have
	\begin{equation} \label{eqn:sob_sob_emb}
		\| \nabla (G*u^\sharp) \|_\infty \leq C \| u^\sharp\|_H.
	\end{equation}
Therefore,
	\[
		\big| \langle B(u,u^\sharp),u^\flat \rangle  \big| \leq \| \nabla u^\flat \|_H \| u^\sharp \|_H \| u \|_H \leq \| u^\flat \|_V \| u^\sharp \|_{D(A)} \| u \|_V,
	\]
	and (\ref{eqn:bilinear_2}) follows. It remains to verify that the condition \eqref{eqn:bilinear_3} is satisfied. For any $u,u^\sharp \in D(A)$ and $u^\flat \in H$, we have $\big| \langle B(u,u^\sharp),u^\flat \rangle  \big|^2 \leq \| B(u,u^\sharp) \|^2_H \| u^\flat \|_H^2$. Since $-\Delta (G*u^\sharp) = u^\sharp$,
	\[
		B(u,u^\sharp) = \nabla \cdot \big( u \nabla (G*u^\sharp) \big) = \nabla u \cdot \nabla(G*u^\sharp) - u \, u^\sharp.
	\]
Hence,
	\begin{align*}
		\| B(u,u^\sharp) \|_H^2 &\leq 2 \| \nabla u \cdot \nabla(G*u^\sharp)\|_H^2 + 2\| u \, u^\sharp \|_H^2 \\
												 &\xupref{eqn:sob_sob_emb}{\leq} 2C\|\nabla u\|_H^2 \| u^\sharp \|_H^2 + 2 \int_{\R^2} u^2 (u^\sharp)^2\dd x \\
												 &\leq 2C\|\nabla u\|_H^2 \| u^\sharp \|_H^2 + 2 \left( \int_{\R^2} u^4 \dd x\right)^{1/2} \left( \int_{\R^2} (u^\sharp)^4 \dd x\right)^{1/2}.
	\end{align*}
Then, using the interpolation inequality $$\int_{\R^2} u^4 \dd x \leq 2 \int_{\R^2}  u^2 \dd x \int_{\R^2}  |\nabla u|^2 \dd x,$$ we get
	\begin{align*}
		\| B(u,u^\sharp) \|_H^2 &\leq 2C \| \nabla u \|_H^2 \| u^\sharp\|_H^2 + 4 \| u \|_H \| \nabla u \|_H \| u^\sharp \|_H \| \nabla u^\sharp\|_H \\
												&\leq C_1 \big( \| \nabla u \|_V^2 \| u^\sharp\|_V^2 +  \| u \|_V \| u \|_V \| u^\sharp \|_V \|  u^\sharp\|_V \\
												&= 2C_1 \| u \|_V^2 \| u^\sharp \|_V^2 \\
												&\leq 2C_1 \| u \|_V \| u \|_{D(A)} \| u^\sharp \|_V \| u^\sharp \|_{D(A)}, 
	\end{align*}
which implies the desired result
	\[
	\big| \langle B(u,u^\sharp),u^\flat \rangle  \big| \leq \| B(u,u^\sharp) \|_H \| u^\flat \|_H \leq C \| u \|_V^{1/2} \| u \|_{D(A)}^{1/2} \| u^\sharp \|_V^{1/2} \| u^\sharp \|_{D(A)}^{1/2} \| u^\flat \|_H.
	\]
The statement of Theorem \ref{thm:strong} now follows from Theorem \ref{thm:temam}.
\end{proof}


\bigskip

\section{Blowup of solutions.}\label{sec:blowup}

In this section we will show that the solutions of both (\ref{eqn:main_eqn_type1}) and (\ref{eqn:main_eqn_type2}) fail to exist globally if the initial data is large enough. In addition, in subection \ref{subsec:blowup for any mass} we show that the solution of (\ref{eqn:main_eqn_type2}) fails to exist globally for any initial data. It is worth mentioning, that in each of these cases, the reason for the lack of global existence will be different, which leads to the need of introducing various types of blowup throughout this section. Namely,
we define three types of blowup.

\begin{definition} \label{def:bu_types}
Let $\xi(\omega) \in [0,+\infty]$ such that $\Prb\big( \xi(\omega)<\infty \big)>0$.
	\begin{enumerate}[label={\bf (\roman{enumi})}] \itemsep=0.8em
		\item $\xi(\omega)$ is a finite time blowup for \eqref{eqn:main_eqn1} of {\bf Type 1} if  
	\begin{equation} \label{eqn:bu_type1}
		\sup_{t \in [0,\xi]} \|\rho\|_{H^1(\R^2)} + \int_0^\xi \| D^2 \rho \|_2 \dd s = +\infty
	\end{equation}
for all $\omega$ such that $\xi(\omega)<\infty$.
		\item $\xi(\omega)$ is a finite time blowup for \eqref{eqn:main_eqn1}  of {\bf Type 2} if 
	\begin{equation} \label{eqn:bu_type2}
		\int_{\R^2} |x|^2 \rho(x,\xi)\dd x = +\infty
	\end{equation}
for all $\omega$ such that $\xi(\omega)<\infty$.
		\item $\xi(\omega)$ is a finite time blowup for \eqref{eqn:main_eqn1}  of {\bf Type 3} if there exists $R>0$ such that
	\begin{equation} \label{eqn:bu_type3}
		\E \| \rho(x,\xi) \|_{L^\infty(B_R)} = +\infty.
	\end{equation}
	\end{enumerate}	
\end{definition}

\subsection{Blowup in the case of large mass for solutions of (\ref{eqn:main_eqn_type1}).}\label{subsec:blowup_1}
In this subsection we show that the solutions of (\ref{eqn:main_eqn_type1}) fail to exist globally for large enough initial conditions. The main result of this subsection is the following Theorem:

\begin{theorem}\label{thm:bu_special}
Assume the initial condition $\rho_0(x)$ is such that 
\[
\int_{\R^2} |x|^2 \rho_0(x) \dd x < \infty
\]
and $m_0 = \int_{\R^2} \rho_0(x) \dd x$
satisfies 
\begin{equation}\label{eqn:large_mass}
\frac{\chi}{2 \pi} m_0 > 2(\nu^2 + \sigma^2).
\end{equation}
Then the solution of (\ref{eqn:main_eqn_type1}) with the initial condition $\rho_0$ has a {\bf Type 3} blowup $\xi(\omega)$. Furthermore, $\xi(\omega)$   satisfies $\xi(\omega) \leq T$ a.s., where $T>0$ is such that
\begin{equation}\label{eqn:cond_on_T}
\left(-2(\nu^2 + \sigma^2) m_0 + \frac{\chi}{2 \pi} m_0^2\right) T - \int_{\R^2}  |x|^2 \rho_0(x) \dd x > 0.
\end{equation}
\end{theorem}

\medskip

Before we prove this theorem we introduce the following definition that will be used in the proof.
\begin{definition}\label{cutoff}
For every $\varepsilon>0$ the function $\varphi_\varepsilon \in C_0^2(\R^2)$ is {\it smooth cutoff function for $|x|^2$} if  $\varphi_\varepsilon(x)=|x|^2$ for $|x| \leq \varepsilon^{-1}$, $\varphi_\varepsilon(x) = 0$ for $|x| \geq 2 \varepsilon^{-1}$, with $|\Delta \varphi_\varepsilon| \leq 4$ and $\nabla \varphi_\varepsilon$ uniformly Lipschitz with Lipschitz constant $2$.
\end{definition}

\medskip

\begin{proof}[Proof of Theorem \ref{thm:bu_special}]
Let $\rho$ be a weak solution of \eqref{eqn:main_eqn_type1} whose local existence is guaranteed by Theorem \ref{thm:exist_loc_spec} for all initial conditions. Using the symmetrization of the double integral, the Definition \ref{def:weak} of the weak solution can be equivalently written in the form (see e.g. \cite{BlanchetDolbeaultPerthame}, Section 2.1)
	\begin{align} \label{def_french}
\nonumber		\int_{\R^2} & \varphi(x)  \rho(x,t) \dd x - \int_{\R^2} \varphi(x) \rho(x,0)\dd x \\
									  & = \int_0^t \left( \int_{\R^2} \frac{\nu^2 + \sigma^2}{2}\Delta \varphi(x)\rho(x,t) \dd x \right. \\
\nonumber									  &\qquad \qquad \left. - \frac{\chi}{4\pi} \int_{\R^2}\! \int_{\R^2} \frac{\big(\nabla \varphi(x)-\nabla\varphi(y)\big)\cdot(x-y)}{|x-y|^2}\rho(x,s)\rho(y,s) \dd x \dd y \right) \dd s \\
																	\nonumber											&\qquad\qquad\qquad\qquad  - \int_0^t \int_{\R^2} \rho(x,s) \nabla \varphi(x) \dd W(s)
	\end{align}
for all $\varphi \in C_0^2(\R^2)$. 

Then by Definition \ref{cutoff} we have
	\begin{gather*}
		\int_{\R^2} |\Delta \varphi_\varepsilon| \rho \dd x \leq 4 m_0, \text{ and} \\
		\int_{\R^2}\! \int_{\R^2} \frac{\big|\nabla \varphi_\varepsilon(x)-\nabla\varphi_\varepsilon(y)\big|\cdot|x-y|}{|x-y|^2}\rho(x,s)\rho(y,s) \dd x \dd y \leq 2 m_0^2.
	\end{gather*}
	
Therefore, noting that $$\E \int_0^t \int_{\R^2} \rho(x,s) \nabla \varphi_\varepsilon(x) \dd x \dd W(s)=0,$$ we get
	\[
		\E \int_{\R^2} \varphi_\varepsilon(x)\rho(x,t)\dd x - \int_{\R^2} \varphi_\varepsilon(x)\rho(x,0)\dd x \leq \int_0^t (4m_0 + 2m_0^2) \dd s = (4m_0 + 2m_0^2)\,t,
	\]
which, in turn implies that $\E \int_{\R^2} \varphi_\varepsilon(x)\rho(x,t) \dd x \leq (4m_0+2m_0^2)\,t +\int_{\R^2} |x|^2 \rho_0(x) \dd x$. Hence, we have
	\[
		\int_{\R^2} \varphi_\varepsilon(x) \rho(x,t)\dd x < \infty \text{ a.s. }
	\]  
Since $\rho(x,t) \geq 0$, we may now apply Fatou's lemma to get
	\[
		\E \int_{\R^2} |x|^2 \rho(x,t) \dd x \leq \liminf_{\varepsilon\to 0} \E \int_{\R^2} \varphi_\varepsilon(x)\rho(x,t)\dd x < \infty.
	\] 
Therefore,
\[
\int_{\R^2} |x|^2 \rho(x,t) \dd x < \infty \text{ a.s. }
\]
Assume $\rho(x,t)$ is defined for all $m_0>0$ and $t\geq 0$. Substituting $\varphi_\varepsilon(x)$ in the definition (\ref{def_french}), and passing to the limit as $\varepsilon \to 0$ in every integral using Lebesgue Dominated Convergence Theorem, we have
	\begin{multline}\label{eqn:eps}
	\lim_{\varepsilon \to 0} \left[\E \int_{\R^2} \varphi_\varepsilon(x) \rho(x,t) \dd x - \int_{\R^2} \varphi_\varepsilon(x) \rho(x,0) \dd x\right] = \\ 
	\E \int_{\R^2} |x|^2 \rho(x,t) \dd x - \int_{\R^2} |x|^2 \rho(x,0) \dd x = \left( 2(\nu^2 + \sigma^2) m_0 - \frac{\chi}{2\pi} m_0^2 \right) t < 0
	\end{multline}
for $m_0$ satisfying (\ref{eqn:large_mass}) and $t\geq T$ satisfying (\ref{eqn:cond_on_T}).  The inequality \eqref{eqn:eps} implies that for sufficiently small $\varepsilon_0>0$ we have
\[
\E \int_{\R^2} \varphi_{\varepsilon_0}(x) \rho(x,t) \dd x < 0.
\]
This contradicts the fact that $\rho(x,t) \geq 0$ a.e. Since $\varphi_{\varepsilon_0}(x)$ has compact support, the solution $\rho(x,t)$ must have a {\bf Type 3} blowup.
\end{proof}

\bigskip
\subsection{Blowup in the case of large mass for solutions of (\ref{eqn:main_eqn_type2}).}\label{subsec:blowup_2}

In this subsection we establish the lack of global existence for the solutions of the equation (\ref{eqn:main_eqn_type2}). Once again, following \cite{BlanchetDolbeaultPerthame}, Section 2.1, we can represent the weak solution of (\ref{eqn:main_eqn_type2}) in the form
	\begin{equation}	 \label{eqn:weak_soln_gen_noise}
	\begin{aligned} 
		\int_{\R^2} & \varphi(x)  \rho(x,t) \dd x - \int_{\R^2} \varphi(x) \rho(x,0)\dd x \\
									  & = \int_0^t \left( \int_{\R^2} \frac{a^2}{2} \Delta \varphi(x)\rho(x,s) \dd x \right. \\
									  &\qquad \left. - \frac{\chi}{4\pi} \int_{\R^2}\! \int_{\R^2} \frac{\big(\nabla \varphi(x)-\nabla\varphi(y)\big)\cdot(x-y)}{|x-y|^2}\rho(x,s)\rho(y,s) \dd x \dd y \right) \dd s \\
																												&\qquad\qquad  +  \int_{\R^2} \int_0^t \Phi\big(\rho(x,s)\big) \varphi(x) \dd W(x,s) \dd x,
	\end{aligned}
	\end{equation}
for all $\varphi \in C_0^2(\R^2)$. As we showed in subsection \ref{subsec:loc_exist_gen_noise}, the solution of \eqref{eqn:weak_soln_gen_noise} exists for $t\in[0,\tau(\omega))$ for some random variable $\tau(\omega) \in (0,+\infty]$.

\smallskip

\begin{remark}
The stopping time $\tau(\omega)$ from \cite{DebGHTem} satisfies \eqref{eqn:bu_type1}; however, we cannot call it a finite time blowup of {\bf Type 1} as it can possibly be infinite with probability 1.
\end{remark}

Before we proceed let us introduce the following notation:
	\begin{alignat*}{2}
		m_0 &\defeq \int_{\R^2} \rho_0(x) \dd x \qquad\qquad M_0 &&\defeq \int_{\R^2} |x|^2 \rho_0(x) \dd x, \\
		m(t) &\defeq \int_{\R^2} \rho(x,t) \dd x \quad\qquad  M(t) &&\defeq \int_{\R^2} |x|^2 \rho(x,t) \dd x.
	\end{alignat*}

\bigskip

\begin{theorem}\label{thm:bu_general}
Suppose there exists $t_1>0$ such that $\Prb\big( \tau(\omega) \geq t_1 \big)=1$. Then for $m_0$ sufficiently large (depending explicitly on $t_1$ and $M_0$) $t_1$ is a {\bf Type 3} blowup.
\end{theorem}
The principal technical difference between Theorem \ref{thm:bu_special} and Theorem \ref{thm:bu_general} is that while the solutions of (\ref{eqn:main_eqn_type1}) satisfy $\int_{\R^2} \rho(x,t) \dd x = m_0$, the quantity $ \int_{\R^2} \rho(x,t) \dd x$ is in fact a random variable for the solutions of (\ref{eqn:main_eqn_type2}). This causes additional complications. In particular, for the proof of Theorem \ref{thm:bu_general} we need the following auxiliary result.

\begin{lemma} \label{lem:finite_mass}
Let $\rho(x,t)$ be a local solution of (\ref{eqn:main_eqn_type2}) with $|\Phi(\rho)|\leq L |\rho|$. For $t_1$ given by Theorem \ref{thm:bu_general}, let $B \defeq \left\{ \omega \in \Omega \colon \int_0^{t_1} \left( \int_{\R^2} \rho(x,s)\dd x \right)^2 \dd s  = +\infty \right\}$. Then $\Prb(B) = 0$.
\end{lemma}

\begin{proof}
It follows from (\ref{eqn:2.18}) that
\begin{equation}\label{eqn:finite}
\E \sup_{t \in [0,t_1]} \|\rho(t,\cdot)\|_{H^{1}(\R^2)}^2 < \infty.
\end{equation}
Then
\[
m^2(t) \leq 2 m_0^2 + 2\left(\sum_{k=1}^{\infty} a_k \int_0^t \int_{\R^2} \Phi(u)e_k(x) \dd x \dd W_k(s) \right)^2
\]
Using the property of stochastic integral 
\[
\E \sup_{t \in [0,T]}\left|\int_0^t f(s) \dd W(s)\right|^2 \leq 4 \int_0^T \E f^2(s) \dd s
\]
as well as \eqref{eqn:finite}, we get
\begin{align*}
\E \sup_{t \in [0,t_1]} m^2(t) &\leq 2 m_0^2 + 8 \sum_{k=1}^{\infty} a_k^2  \int_0^{t_1} \E\left(\int_{\R^2} \Phi(\rho)e_k(x) \dd x \right)^2 \dd s \\
 &\leq 2 m_0^2 + 8 \sum_{k=1}^{\infty} a_k^2  \int_0^{t_1} \left(\E \int_{\R^2} L^2 |\rho(s,x)|^2 \dd x \int_{\R^2} e_k^2(x) \dd x \right)  \dd s  \\
 &\leq 2 m_0^2 + 8 \sum_{k=1}^{\infty} a_k^2 L^2 t_1 \E \sup_{t \in [0,t_1]} \|\rho(t, \cdot)\|_{H^1(\R^2)}^2 < \infty.
\end{align*}
Therefore, 
\[
\E \int_{0}^{t_1} m^2(t) \dd t  = \int_{0}^{t_1} \E m^2(t) \dd t \leq t_1 \, \E \sup_{t \in [0,t_1]} m^2(t)  < \infty, 
\]
which implies that 
\[
\int_{0}^{t_1} m^2(t) \dd t < \infty \text{ a.s. }
\]
and the statement of the Lemma follows.
\end{proof}

\medskip

\begin{proof}[Proof of Theorem \ref{thm:bu_general}.]
Consider two cases:

\medskip

\noindent{\emph{Case 1:}} Suppose 
\begin{equation}\label{eqn:inf_exp}
\E \int_0^{t_1} \left( \int_{\R^2} \rho(x,s) \dd x \right)^2 \dd s = +\infty.
\end{equation}
By Lemma \ref{lem:finite_mass}, $\int_0^{t_1} \left( \int_{\R^2} \rho(x,s) \dd x \right)^2 < +\infty$ a.s. For $\varepsilon>0$, let $\varphi_\varepsilon$ be given by Definition \ref{cutoff}. Introduce the next function
	\[
		\rho_\varepsilon(x,s) = \begin{dcases*}
						 \rho(x,s) &  if $|x| \leq \frac{1}{\varepsilon}$, \\
						 0	&  if $|x|>\frac{1}{\varepsilon}$.
					\end{dcases*}
	\]
Then
	\[
		\int_0^{t_1} \int_{\R^2} \! \int_{\R^2} \frac{\nabla \varphi_\varepsilon(x) - \nabla\varphi_\varepsilon(y)}{|x-y|^2} \cdot (x-y) \rho_\varepsilon(x,s)\rho_\varepsilon(y,s) \dd x \dd y \dd s = 2 \int_0^{t_1} \left( \int_{\R^2} \rho_\varepsilon(x,s)\dd x \right)^2 \dd s.
	\]
On one hand, by Fatou's lemma and (\ref{eqn:inf_exp})
	\[
		\liminf_{\varepsilon\to 0} \E \int_0^{t_1} \left( \int_{\R^2} \rho_\varepsilon(x,s)\dd x \right)^2 \dd s \geq \E \int_0^{t_1} \left( \int_{\R^2} \rho(x,s)\dd x \right)^2 \dd s = +\infty.
	\]
On the other hand,
	\[
		\int_0^{t_1} \int_{\R^2} \! \int_{\R^2} \frac{\nabla \varphi_\varepsilon(x) - \nabla\varphi_\varepsilon(y)}{|x-y|^2} \cdot (x-y) \Big(\rho_\varepsilon(x,s)\rho_\varepsilon(y,s) - \rho(x,s)\rho(y,s)  \Big) \dd x \dd y \dd s \to 0 
	\]
a.s. as $\varepsilon \to 0$, since
	\begin{multline} \nonumber
	\int_0^{t_1} \int_{\R^2} \! \int_{\R^2} \frac{\nabla \varphi_\varepsilon(x) - \nabla\varphi_\varepsilon(y)}{|x-y|^2} \cdot (x-y) \Big(\rho_\varepsilon(x,s)\rho_\varepsilon(y,s) - \rho(x,s)\rho(y,s)  \Big) \dd x \dd y \dd s \\
		\leq 2 \int_0^{t_1} \int_{\R^2}\! \int_{\R^2} \Big( \rho(x,s)\rho(y,s)-\rho_\varepsilon(x,s)\rho_\varepsilon(y,s) \Big)\dd x \dd y \dd s
	\end{multline}
with $0 \leq \rho(x,s)\rho(y,s)-\rho_\varepsilon(x,s)\rho_\varepsilon(y,s) \leq \rho(x,s)\rho(y,s)$ and $\int_0^{t_1} \int_{\R^2} \! \int_{\R^2} \rho(x,s)\rho(y,s) \dd x \dd y < +\infty$ a.s. Since 
$$|\Delta \varphi_\varepsilon(x) \rho(x,s)| \leq 4 \rho(x,s)$$ 
and 
$$4\E \int_{\R^2}\rho(x,s)\dd x = 4 m_0<\infty, $$ by Lebesgue Dominated Convergence Theorem,
	\[
		\E \int_0^{t_1} \int_{\R^2} \Delta \varphi_\varepsilon(x) \rho(x,s) \dd x \dd s \to 4 t_1 m_0, \ \ \ \varepsilon\to 0.
	\]

Hence, taking the expected value of both sides in \eqref{eqn:weak_soln_gen_noise}, we obtain that there exists $\varepsilon_0>0$ such that for all $\varepsilon<\varepsilon_0$
	\[
		\E \int_{\R^2} \varphi_\varepsilon(x) \rho(x,t_1 ) \dd x < 0,
	\]
and since $\supp \varphi_\varepsilon \subset B_{1/\varepsilon}$, we conclude that $\E \| \rho(x,t_1) \|_{L^\infty(B_{2/\varepsilon_0})} = +\infty$.

\medskip

\noindent\emph{Case 2:} Suppose 
\begin{equation}\label{eqn:finite2}
\E \int_0^{t_1} \left( \int_{\R^2} \rho(x,s) \dd x \right)^2 \dd s < +\infty.
\end{equation} 
Since
	\[
		\frac{\big| \nabla \varphi_\varepsilon(x) - \nabla \varphi_\varepsilon(y) \big|}{|x-y|} \rho(x,s) \rho(y,s) \leq 2 \rho(x,s)  \rho(y,s),
	\]
using (\ref{eqn:finite2}) and Lebesgue Dominated Convergence Theorem, we may pass to the limit as $\varepsilon\to 0$ to obtain
	\begin{align*}
		\E \int_0^{t_1} \int_{\R^2}\!\int_{\R^2} &\frac{\nabla \varphi_\varepsilon(x) - \nabla \varphi_\varepsilon(y)}{|x-y|^2}\cdot (x-y) \rho(x,s) \rho(y,s) \dd x \dd y \dd s \\
															&\to 2  \E \int_0^{t_1} \left( \int_{\R^2} \rho(x,s) \dd x \right)^2 \dd s \geq 2 \int_0^{t_1} \left( \E \int_{\R^2} \rho(x,s)\dd x \right)^2 \dd s \\
															&=2 \int_0^{t_1} m_0^2 \dd s = 2 m_0^2 t_1.
	\end{align*}
Therefore, for $\varepsilon \leq \varepsilon_0$, we have 
	\[
		\E \int_{\R^2} \varphi_\varepsilon(x) \rho(x,t) \dd x \leq M_0 - \frac{\chi}{2\pi}t_1 m_0^2 + 2 a^2 m_0 t_1 <0,
	\]
and, once again, $\E \| \rho(x,t_1) \|_{L^\infty(B_{2/\varepsilon_0})} = +\infty$.
\end{proof}

\medskip

\subsection{Blowup for any mass for solutions of (\ref{eqn:main_eqn_type2}).}\label{subsec:blowup for any mass}
In this section we show that in the case of linear multiplicative noise, the equation (\ref{eqn:main_eqn_type2}) has a blowup with nonzero probability for any nontrivial initial conditions. In other words, we show that any nontrivial local solution cannot be extended globally.

Recall that by Proposition \ref{prop:max_pr} we have $\E m(t) = m_0$ for $t\geq 0$ within the interval of existence. We consider (\ref{eqn:main_eqn_type2}) with $\Phi(\rho) := \sigma\rho$ for some $\sigma \in \R$, and $W(x,s):=W(s)$.

\begin{theorem}\label{thm:bu_type1_2}
Let $\rho(x,t)$ be a solution of (\ref{eqn:main_eqn_type2}) with an initial condition $\rho_0\in H^1(\R^2)$, $\rho_0 \not \equiv 0$ and $\int_{\R^2}|x|^2\rho_0(x)\dd x<+\infty$. Then there exists $t_2=t_2(m_0)<\infty$ and a random variable $\xi(\omega)$ such that $\xi(\omega)$ is a finite time blowup for $\rho(x,t)$ of either {\bf Type 1} or {\bf Type 2}, and $p_2 \defeq \Prb\big( \xi(\omega) < t_2 \big) >0,$ with $p_2$ determined explicitly in terms of $m_0$.
\end{theorem}

\begin{proof}
It follows from Theorem \ref{thm:strong} that the local solution of (\ref{eqn:main_eqn_type2}) is a strong solution. Hence we can integrate (\ref{eqn:main_eqn_type2}) both in $x$ and in $t$ and thus obtain the following ordinary linear stochastic equation for the function $m(t)$:
	\begin{equation} \label{eqn:m}
		m(t) - m_0 = \sigma \int_0^t m(s) \dd W(s).
	\end{equation}
Let $t$ be a Markov's moment in \eqref{eqn:m} with respect to the filtration $\mathcal{F}_s^W$. Note that \eqref{eqn:m} is a linear SDE of the form $\dd m(t) = \sigma m(t)\dd W(t)$, which, using Ito's formula, has the explicit solution
	\begin{equation} \label{eqn:m_soln}
		m(t) = m_0 \exp \left( \frac{\sigma^2}{2} t + \sigma W(t) \right).
	\end{equation}
Next, we show that $M(t) = \int_{\R^2} |x|^2 \rho(x,t) \dd x < \infty$ a.s. for $t \leq \tau(\omega)$, where $\tau(\omega)$ is defined as in Theorem \ref{thm:strong}. Using $\varphi_\varepsilon$ given by Definition \ref{cutoff}, we have
	\begin{align*}
		-\int_{\R^2} \! \int_{\R^2}   \frac{\big(\nabla \varphi_\varepsilon(x)-\nabla\varphi_\varepsilon(y)\big)\cdot(x-y)}{|x-y|^2}& \rho(x,s)\rho(y,s) \dd x \dd y \\
										   &\leq \int_{\R^2} \! \int_{\R^2}  \frac{\big|\nabla \varphi_\varepsilon(x)-\nabla\varphi_\varepsilon(y)\big|}{|x-y|}\rho(x,s)\rho(y,s) \dd x \dd y \\
										   &\leq 2 \left( \int_{\R^2} \rho(x,s)\dd x\right)^2 = 2 m^2(s),
	\end{align*}
and $\int_{\R^2} \Delta \varphi_\varepsilon(x)\rho(x,s)\dd x \leq 4 m(s).$ Now define
	\begin{align*}
		u_\varepsilon(s) &\defeq \int_{\R^2} \rho(x,s)\varphi_\varepsilon(s) \dd x \\
		\psi_\varepsilon(s) &\defeq -\frac{\chi}{4\pi} \int_{\R^2} \!\int_{\R^2} \frac{\big(\nabla \varphi_\varepsilon(x)-\nabla\varphi_\varepsilon(y)\big)\cdot(x-y)}{|x-y|^2} \rho(x,s)\rho(y,s) \dd x \dd y \\
									&\qquad\qquad\qquad\qquad\qquad\qquad\qquad\qquad\qquad\qquad + \frac{a^2}{2} \int_{\R^2} \Delta \varphi_\varepsilon(x)\rho(x,s) \dd x.
	\end{align*}
Then $u_\varepsilon$ satisfies the SDE
	\begin{equation} \label{eqn:u_eps}
				 \left\{ \begin{array}{rl}
													\dd u_\varepsilon(t) &= \psi_\varepsilon(t)\dd t + \sigma u_\varepsilon(t)\dd W(t)\\  
													u_\varepsilon(0)		 &= \int_{\R^2} \rho_0(x)\varphi_\varepsilon(x)\dd x. 
				\end{array}\right.
	\end{equation}
Let $u^+$ solve
	\begin{equation} \label{eqn:u_plus}
				 \left\{ \begin{array}{rl}
													\dd u^+(t) &= \left(2 a^2 m(t) + \frac{\chi}{2\pi}m^2(t) \right)\dd t + \sigma u^+(t) \dd W(t)\\  
													u^+(0)		 &= M_0. 
				\end{array}\right.
	\end{equation}
Since $u^+(0) \leq u_\varepsilon(0)$ and $\psi_\varepsilon(t) \leq 2 a^2 m(t) + \frac{\chi }{2\pi}m^2(t)$, then using the stochastic comparison theorem by Watanabe and Ikeda\cite{WatIke}, we have
	\begin{equation} \label{eqn:stoch_comp}
		0 \leq u_\varepsilon(t) \leq u^+(t) \qquad \text{ for all } t \geq 0.
	\end{equation}
Note that
	\[
		u^+(t) = \Psi(t) \left( M_0 + \int_0^t \left( 2a^2m(s)+\frac{\chi}{2\pi}m^2(s) \right)\Psi^{-1}(s)\dd s \right),
	\]
where $\Psi(t) = \exp \left( -\frac{\sigma^2}{2}t + \sigma W(t) \right)$. Then, using \eqref{eqn:m_soln}, we have
	\[
		\int_0^t \left(2 a^2m(s)+\frac{\chi}{2\pi}m^2(s) \right)\Psi^{-1}(s)\dd s = 2a^2 m_0 t + \frac{\chi}{2\pi}m_0^2 \int_0^t  \exp \left( -\frac{\sigma^2}{2}s + \sigma W(s) \right) \dd s.
	\]
Hence,
	\begin{multline}\label{eqn:u+}
		u^+(t)  = \exp \left( -\frac{\sigma^2}{2}t + \sigma W(t) \right)\\ \times \left( M_0 + 2a^2 m_0 t + \frac{\chi m_0^2}{2\pi} \int_0^t \exp \left( -\frac{\sigma^2}{2}s + \sigma W(s) \right) \dd s \right) < \infty
	\end{multline}
for all $t \geq 0$. Passing to the limit $\varepsilon \to 0$ in \eqref{eqn:stoch_comp}, we have
	\[
		M(t) = \lim_{\varepsilon\to 0} u_\varepsilon(t) \leq u^+(t) < \infty.
	\]
Let $u(s) = \int_{\R^2} |x|^2 \rho(x,s)\dd x$. Since $(u_\varepsilon(s) - u(s))^2 \leq 2 (u^+(s))^2$ and $$\int_0^t (u^+(s))^2 \dd s<\infty,$$ by Lebesgue Dominated Convergence Theorem, we have $$\int_0^t (u_\varepsilon(s) - u(s))^2 \dd s \to 0$$ in probability. This, in turn, implies that
	\[
		\int_0^t \int_{\R^2} \rho(x,s)\varphi_\varepsilon(x) \dd x \dd W(s) \to \int_0^t \int_{\R^2} |x|^2 \rho(x,s) \dd x \dd W(s) \quad \text{a.s.}.
	\]
Then we can pass to the limit $\varepsilon\to 0$ in \eqref{eqn:weak_soln_gen_noise}, and get
	\begin{equation} \label{eqn:M}
		M(t) - M_0 = \int_0^t \left( 2a^2 m(s) - \frac{\chi}{2\pi} m^2(s) \right)\dd s + \sigma \int_0^t M(s) \dd W(s).
	\end{equation}
Therefore we get that
	\[
	M(t) = \exp \left( -\frac{\sigma^2}{2}t + \sigma W(t) \right) \left( M_0 + 2a^2 m_0 t - \frac{\chi m_0^2}{2\pi} \int_0^t \exp \left( -\frac{\sigma^2}{2}s + \sigma W(s) \right) \dd s \right).
	\]
Now, fix $\alpha, \beta>0$. For $t>0$, we define the random set
	\[
		A_{\alpha,\beta}^t \defeq \left\{ \omega \in \Omega \colon W(s) \geq \left( \frac{\sigma}{2} + \alpha \right) s - \beta, \ s\in[0,t] \right\}.
	\]
Then for every $\omega \in A_{\alpha,\beta}^t$ and for every $s\in[0,t]$ we have $-\frac{\sigma^2}{2}s + \sigma W(s) \geq \sigma \alpha s - \sigma \beta$. This implies that $$\int_0^t \exp \left( -\frac{\sigma^2}{2}s + \sigma W(s) \right)\dd s \geq \frac{1}{\sigma \alpha} e^{-\sigma \beta} \left( e^{\sigma\alpha t}-1 \right),$$ and we get
	\[
		M(t) \leq \Psi(t) \left( M_0 + 2a^2 m_0 t + \frac{\chi m_0^2}{\sigma\alpha e^{\sigma\beta} 2\pi} - \frac{\chi m_0^2}{\sigma \alpha 2 \pi}e^{\sigma(\alpha t - \beta)}\right).
	\]
Let $t_2$ be such that
	\begin{equation}\label{eqn:t_2}
		M_0 + 2a^2 m_0 t_2 + \frac{\chi m_0^2}{\sigma\alpha e^{\sigma\beta} 2\pi} \leq  \frac{\chi m_0^2}{\sigma \alpha 2 \pi}e^{\sigma(\alpha t_2 - \beta)}.
	\end{equation}
Then $M(t_2) \leq 0$, which contradicts the positivity of $\rho(x,t)$. Thus, 
$$p_{\alpha,\beta}^{t_2} \defeq \Prb(A_{\alpha,\beta}^{t_2}) >0$$ is the probability of {\bf Type 2} blowup under the condition that $\tau(\omega)>t_2$. 

We are now ready to estimate the probability of {\bf Type 1} or {\bf Type 2} blowup for $\rho(x,t)$. Let
	\[
		B_{t_2} \defeq \{ \omega \in \Omega \colon \tau(\omega) \leq t_2 \} \qquad \text{ with } \quad  \Prb(B_{t_2}) = z \in [0,1].
	\]
Clearly $\omega \in B_{t_2}$ implies {\bf Type 1} blowup. Let $I$ and $II$ denote the events of {\bf Type 1} and {\bf Type 2} blowup, respectively. Then 
	\begin{equation} \label{eqn:prob_type_1_2}
		\Prb(I \cup II) \geq \Prb(B_{t_2}) = z.
	\end{equation}
On the other hand,
	\begin{equation} \label{eqn:prob_type_1_2_2}
		\begin{aligned}
			\Prb( I \cup II) &\geq \Prb(A_{\alpha,\beta}^{t_2} \cap \ol{B_{t_2}}) = \Prb(A_{\alpha,\beta}^{t_2}) + \Prb(\ol{B_{t_2}})-P(A_{\alpha,\beta}^{t_2} \cup \ol{B_{t_2}}) \\
								&\geq \Prb(A_{\alpha,\beta}^{t_2}) + \Prb(\ol{B_{t_2}}) - 1 = p_{\alpha,\beta}^{t_2} + (1-z) -1 \\
								&= p_{\alpha,\beta}^{t_2} - z.
		\end{aligned}
	\end{equation}
Combining \eqref{eqn:prob_type_1_2} and \eqref{eqn:prob_type_1_2_2}, we have $\Prb(I \cup II) \geq \max \{z , p_{\alpha,\beta}^{t_2} - z \}$. Since the right-hand side is minimal if $z= \frac{1}{2}p_{\alpha,\beta}^{t_2}$, we have that $\Prb(I \cup II) \geq \frac{1}{2}p_{\alpha,\beta}^{t_2} > 0$.
\end{proof}

\bigskip
\section{Small perturbation}\label{sec:small_perturb}

In this section we consider the limit behavior of solutions of (\ref{eqn:main_eqn1}) as the noise contribution vanishes. To this end, for small $\varepsilon>0$ we consider
	\begin{equation} \label{eqn:eps_SPDE}
				 \left\{ \begin{array}{rl}
													\dd \rho_\varepsilon &= \Big( \frac{a^2}{2} \Delta \rho_\varepsilon - \chi \nabla \cdot \big( \rho_\varepsilon \nabla (G*\rho_\varepsilon) \big) \Big) \dd t + \varepsilon\,\nabla \rho_\varepsilon \dd W(t) \\  
													\rho_\varepsilon(0)		 &= \rho_0, 
				\end{array}\right.
	\end{equation}
and
	\begin{equation} \label{eqn:limit_PDE}
				 \left\{ \begin{array}{rl}
													\dd \rho_* &= \Big( \frac{a^2}{2} \Delta \rho_*) - \chi \nabla \cdot \big( \rho_* \nabla (G*\rho_*) \big) \Big) \dd t \\  
													\rho_* (0)		 &= \rho_0. 
				\end{array}\right.
	\end{equation}
	Assume
	\begin{equation*}\label{eqn:smallness_cond}
	 C \left( \frac{-\nu^2 p (p-1)}{2} + \chi \right) m_0^{-p/(p-1)} + \chi m_0^{(p-2)/(p-1)} < 0 \qquad \text{ for } p \geq 2,
	\end{equation*}
	with $C$ defined in Theorem \ref{th:global}. Hence, if $\rho_\varepsilon$ and $\rho_{*}$ are two global solutions of (\ref{eqn:eps_SPDE}) and (\ref{eqn:limit_PDE}) respectively, for $\varepsilon$ small enough, these solutions are defined globally as a consequence of Theorem \ref{th:global}.
The main result of this Section is the following Theorem.
\begin{theorem}
Let $\rho_0 \in H^s(\R^2) \cap L^1(\R^2) \cap L^p(\R^2)$ with $s>3$ and $p \geq 2$.  Then for any $T>0$ we have

	\begin{equation}\label{eqn:u_eps_conv}
		\E \sup_{t\in[0,T]} \| \rho_\varepsilon(t) - \rho_*(t) \|_p^p \to 0 \quad \text{ as } \varepsilon\to 0.
	\end{equation}
	
	\end{theorem}
\begin{proof}
Using \cite[Theorem 1.7]{LiRodZha}, we obtain $\rho_* \in C([0,\infty); H^s(\R^2) \cap L^1(\R^2)).$ In particular,
	\begin{equation}\label{eqn:sup_bd_u0}
		\sup_{t\in[0,T]} \big\{ \| \rho_*(t) \|_2, \, \| \nabla \rho_*(t)\|_2 \big\} \leq A
	\end{equation}
for some $A>0$. Let $y_\varepsilon \defeq \rho_\varepsilon - \rho_*$. Then,
	\begin{equation} \label{eqn:y_eps_SPDE}
				 \left\{ \begin{array}{rl}
													\dd y_\varepsilon &= \Big( \frac{a^2}{2} \Delta y_\varepsilon - \chi \nabla \cdot \big( \rho_\varepsilon \nabla (G*\rho_\varepsilon) - \rho_* \nabla(G*\rho_*)\big) \Big) \dd t + \varepsilon\,\nabla \rho_\varepsilon \dd W(t) \\  
													y_\varepsilon(0)		 &= 0. 
				\end{array}\right.
	\end{equation}
Note that, as we showed in subsection \ref{subsec:global_exist}, $\esssup_{\omega \in \Omega} \sup_{t\in [0,T]} \|\rho_\varepsilon\|_p \leq \| \rho_0 \|_p$. Now, Ito's formula for $\| y_\varepsilon(t) \|_p^p$ yields
	\begin{equation}	 \label{eqn:Ito_y_eps}
	\begin{aligned} 
		\| y_\varepsilon \|_p^p &= -p(p-1)\frac{a^2}{2} \int_0^t \int_{\R^2} y_\varepsilon^{p-2} |\nabla y_\varepsilon|^2 \dd x \dd s \\
										 &\qquad\qquad - \chi p \int_0^t \int_{\R^2} y_\varepsilon^{p-1} \nabla \cdot \big( \rho_\varepsilon\nabla (G*\rho_\varepsilon) - \rho_* \nabla(G*\rho_*) \big)\dd x \dd s \\
										 &\qquad\qquad\qquad\qquad +\frac{\varepsilon^2}{2} \int_0^t \int_{\R^2} p(p-1) y_\varepsilon^{p-2} |\nabla \rho_\varepsilon|^2 \dd x \dd s \\
										 &\qquad\qquad\qquad\qquad\qquad\qquad + \varepsilon \int_0^t \int_{\R^2} p y_\varepsilon^{p-1} \nabla \rho_\varepsilon \dd x \dd W(s) \\
										 &= I_1 + I_2 + I_3 + I_4.
	\end{aligned}
	\end{equation}
Next we will estimate each $I_1,\ldots,I_4$ in \eqref{eqn:Ito_y_eps}.
Note that
	\begin{equation} \label{eqn:I_1}
		I_1 = -a^2 p(p-1)\frac{2}{p^2} \int_0^t \int_{\R^2} \left|\nabla \left(y_\varepsilon^{p/2}\right)\right|^2 \dd x \dd t.
	\end{equation}
	
For the second term, noting that the nonlocal term can be written as $\nabla \cdot \Big( y_\varepsilon \nabla(G*\rho_\varepsilon)+\rho_* \nabla (G*y_\varepsilon) \Big)$, and using the bound $\esssup_{\omega \in \Omega} \sup_{t\in [0,T]} \|\rho_\varepsilon\|_p \leq \|\rho_0\|_p$, we estimate
	\begin{equation} \label{eqn:I_2_part1}
		\begin{aligned}
			\bigg| \chi \int_{\R^2} p y_{\varepsilon}^{p-1} \nabla \cdot & \left( y_\varepsilon \nabla (G * \rho_\varepsilon) \right) \dd x \bigg| \\
																			   & = \chi p (p-1) \left| \int_{\R^2} y_\varepsilon^{p-1} \nabla(G*\rho_\varepsilon) |\nabla y_\varepsilon|^2  \dd x \right| \\
																			   &\leq \chi p (p-1) \left( \int_{\R^2} |y_\varepsilon^{p/2 - 1} \nabla y_\varepsilon|^2 \dd x \right)^{1/2} \left(  y_\varepsilon^p |\nabla (G*\rho_\varepsilon)|^2 \dd x  \right)^{1/2} \\
																			   &\leq \chi \frac{2(p-1)}{p} \left( \int_{\R^2} |\nabla y_\varepsilon^{p/2} |^2 \dd x \right)^{1/2} \left(  y_\varepsilon^p |\nabla (G*\rho_\varepsilon)|^2 \dd x  \right)^{1/2} \\
																			   &\leq \alpha 2\chi(p-1) \| \nabla y_\varepsilon^{p/2} \|_2^2 + \frac{1}{\alpha} 2 \chi (p-1) C_p^2 \| \rho_\varepsilon\|_p^2 \int_{\R^2} y_\varepsilon^p \dd x \\
																			   &\leq \alpha 2 \chi (p-1)  \| \nabla y_\varepsilon^{p/2} \|_2^2 + \frac{1}{\alpha} 2 \chi (p-1) C_p^2 \| \rho_0\|_p^2 \int_{\R^2} y_\varepsilon^p \dd x.
		\end{aligned}
	\end{equation}
Likewise, we have
	\begin{equation} \label{eqn:I_2_part2}
		\begin{aligned}
			\bigg| \chi \int_{\R^2} p y_{\varepsilon}^{p-1} \nabla \cdot & \left( \rho_* \nabla (G * y_\varepsilon) \right) \dd x \bigg| \\
																			   & = \chi p (p-1) \left| \int_{\R^2} y_\varepsilon^{p-2}\rho_* \nabla(G*y_\varepsilon) \cdot \nabla y_\varepsilon  \dd x \right| \\
																			   &\leq \chi p (p-1) \left( \int_{\R^2} |\nabla y_\varepsilon^{p/2} |^2\dd x \right)^{1/2} \left(  y_\varepsilon^{p-2}\rho_*^2 |\nabla (G*y_\varepsilon)|^2 \dd x  \right)^{1/2} \\
																			   &\leq \frac{\alpha 2\chi(p-1)}{p} \int_{\R^2} \| \nabla y_\varepsilon^{p/2} \|^2 \dd x + \frac{1}{\alpha} 2 \chi p(p-1) C_p \|y_\varepsilon\|_p^2 \int_{\R^2} y_\varepsilon^{p-2} \rho_*^2 \dd x  \\
																			   &\leq \frac{\alpha 2 \chi (p-1)}{p} \int_{\R^2} \| \nabla y_\varepsilon^{p/2} \|^2 \dd x + \frac{C_p \chi p (p-1)}{\alpha} \left( \int_{\R^2} y_\varepsilon^p \dd x \right)^{\frac{p-2}{p}} \times \\
																			   & \qquad\qquad\qquad\qquad\qquad\qquad\qquad\qquad \times \left( \int_{\R^2} \rho_*^p \dd x \right)^{\frac{2}{p}} \left( \int_{\R^2} y_\varepsilon^p \dd x \right)^{\frac{2}{p}} \\
																			   &\leq \frac{\alpha 2 \chi (p-1)}{p} \int_{\R^2} \| \nabla y_\varepsilon^{p/2} \|^2 \dd x + \frac{C_p \chi p (p-1)}{\alpha} \|\rho_0\|_p^2 \left( \int_{\R^2} y_\varepsilon^p \dd x \right)^{\frac{2}{p}}.
		\end{aligned}
	\end{equation}		
Then putting together \eqref{eqn:I_2_part1} and \eqref{eqn:I_2_part2}, for the second term in \eqref{eqn:Ito_y_eps} we obtain
	\begin{equation} \label{eqn:I_2}
		|I_2| \leq \int_0^t \alpha \left(2\chi (p-1) + \frac{2\chi(p-1)}{p}\right) \| \nabla y_\varepsilon^{p/2} \|_2^2 \dd s + C \int_0^t \int_{\R^2} y_\varepsilon^p \dd x \dd s.
	\end{equation}
	
For the third term in \eqref{eqn:Ito_y_eps}, note that
	\begin{equation} \label{eqn:third_term}
	\begin{aligned}
		\frac{\varepsilon^2}{2} p (p-1) \int_{\R^2} y_\varepsilon^{p-2} & |\nabla \rho_\varepsilon|^2 \dd x \\
																								 & \leq \varepsilon^2 p (p-1) \int_{\R^2} y_\varepsilon^{p-2} |\nabla y_\varepsilon|^2 \dd x + \varepsilon^2 p (p-1) \int_{\R^2} y_\varepsilon^{p-2} |\nabla \rho_*|^2 \dd x.
	\end{aligned}
	\end{equation}
The first term on the right-hand side will be compensated by $I_2$.  To estimate $\int_{\R^2} y_\varepsilon^{p-2}|\nabla \rho_*|^2 \dd x$, first note that $\| y_\varepsilon\|_p \leq \|\rho_\varepsilon\|_p + \|\rho_* \|_p \leq 2\| \rho_0\|_p$. Then, by the embedding $W^{3+\alpha,2}(\R^2) \subset C^{2,\alpha}(\R^2)$ with $\alpha \in (0,1]$, we have 
	\begin{equation} \label{eqn:grad_u0}
		\| \nabla \rho_* \|_\infty \leq C \sup_{t\in [0,T]} \| \rho_*(t) \|_{W^{3+\alpha,2}(\R^2)} \leq A,
	\end{equation}
where the constant $A>0$ is given in \eqref{eqn:sup_bd_u0}. Furthermore, using the interpolation inequality,
	\begin{equation} \label{eqn:y_eps_interpol}
		\| y_\varepsilon \|_{p-2}^{p-2} \leq \| y_\varepsilon \|_{p}^{\frac{2p}{p-1}} \| y_\varepsilon \|_1^{\frac{2}{p-1}} \leq C \esssup_{\omega \in \Omega} \sup_{t\in[0,T]} \| y_\varepsilon \|_p^{\frac{2p}{p-1}} \leq C \| \rho_* \|_p^{\frac{2p}{p-1}}.
	\end{equation}
Combining these, we get
	\begin{equation} \label{eqn:I_3}
		|I_3| \leq \varepsilon^2 p(p-1) \frac{4}{p} \int_0^t \int_{\R^2} |\nabla y_\varepsilon^{p/2}|^2 \dd x \dd s + C\,T\,\varepsilon^2.
	\end{equation}
	
Now we estimate the stochastic term $I_4$. Note that $\int_{\R^2} y_\epsilon^{p-1} \nabla \rho_\varepsilon \dd x = \int_{\R^2} y_\varepsilon^{p-1} \nabla \rho_* \dd x$ since $\int_{\R^2} y_\varepsilon^{p-1} \nabla y_\varepsilon \dd x = 0$. Then
	\begin{align*}
		\E \sup_{t\in[0,T]} |I_4| &\leq \varepsilon p \,\E \sup_{t\in[0,T]} \left| \int_0^t \int_{\R^2} y_\varepsilon^{p-1} \nabla \rho_* \dd x \dd W(s) \right| \\
										 &\leq \varepsilon p \left( \E \left(\sup_{t\in [0,T]} \left| \int_0^t \int_{\R^2} y_\varepsilon^{p-1} \nabla \rho_* \dd x \dd W(s)\right|\right)^2 \right)^{1/2} \\
										 &\leq \varepsilon p \left( \E \int_0^t \E |\int_{\R^2} y_\varepsilon^{p-1} \nabla \rho_* \dd x|^2 \dd s \right)^{1/2}.
	\end{align*}	 
Proceeding similar as in the estimate for $I_3$ we obtain that
	\begin{equation} \label{eqn:I_4}
		\E \sup_{t\in [0,T]} |I_4(t)| \leq C \varepsilon.
	\end{equation}
Going back to \eqref{eqn:Ito_y_eps} and putting together \eqref{eqn:I_1}, \eqref{eqn:I_2}, and \eqref{eqn:I_3}, we get
	\begin{equation} \label{eqn:Ito_final_estimate}
		\begin{aligned}
			\| y_\varepsilon(t) \|_p^p &\leq \left( -2a^2\frac{p-1}{p} + \alpha 2 (p-1) + \frac{2\chi (p-1)}{p} \right) \int_0^t \| \nabla y_\varepsilon^{p/2}(s) \|_2^2 \dd s \\
												& \qquad\qquad + C_1 \int_0^t \| y_\varepsilon(s) \|_p^p \dd s + \frac{4\varepsilon^2 (p-1)}{p} \int_0^t \| \nabla y_\varepsilon^{p/2}(s) \|_2^2 \dd s \\
												& \qquad\qquad\qquad\qquad +C_2 \, T \varepsilon^2 + \sup_{t\in[0,T]} \varepsilon p \left| \int_0^t \int_{\R^2} y_\varepsilon^{p-1} \nabla \rho_\varepsilon \dd x \dd W(s) \right| \\
												& \leq C_3 \int_0^t \| y_\varepsilon(s) \|^p \dd s + C_2 \, T \varepsilon^2 + \varepsilon p \sup_{t\in[0,T]}  \left| \int_0^t \int_{\R^2} y_\varepsilon^{p-1} \nabla \rho_\varepsilon \dd x \dd W(s) \right|.
		\end{aligned}
	\end{equation}
For $\alpha$ and $\varepsilon$ sufficiently small, Gr\"onwall's inequality implies that
	\[
		\| y_\varepsilon(t) \|_p^p \leq \left( C_2 \varepsilon^2 T  + \varepsilon p \sup_{t\in[0,T]} \left| \int_0^t \int_{\R^2} y_\varepsilon^{p-1}\nabla \rho_\varepsilon \dd x \dd W(s) \right|\right)\,e^{C\,T}.
	\]
Finally, taking expectations and using the estimate \eqref{eqn:I_4}, we obtain $$\E \sup_{t\in [0,T]} \| y_\varepsilon(t) \|_p^p \to 0, \ \ \varepsilon \to 0,$$ which yields the desired result.
\end{proof}

\bigskip
\subsection*{Acknowledgments}
The research of Oleksandr Misiats was supported by Simons Collaboration Grant for Mathematicians No. 854856. The research of Ihsan Topaloglu was supported by Simons Collaboration Grant for Mathematicians No. 851065. The research of Oleksandr Stanzhytskyi was supported by Ukrainian Government Scientific Research Grant No. 210BF38-01.

\bibliographystyle{IEEEtranS}
\def\url#1{}
\bibliography{references}

\end{document}